\documentclass[fleqn,reqno,11pt,twoside]{article}
\usepackage[hmargin=1.25in, vmargin=1.25in, marginparwidth=1in, marginparsep=0.1in, a4paper, centering]{geometry}
\usepackage{amsmath, amsthm, amssymb, amsfonts}
\usepackage{mathtools} 
\usepackage[english]{babel}
\usepackage{cite}
\usepackage{fourier}
\usepackage[lining]{ebgaramond}
\usepackage{nicefrac}
\usepackage{enumerate}
\usepackage{pifont}
\usepackage{booktabs}  
\usepackage{texlogos}  
\usepackage{amsfonts}
\usepackage{upgreek}
\usepackage{xcolor}
\usepackage{graphicx}
\usepackage{comment}  
\usepackage{paralist}  
\usepackage[font=small, labelfont=bf, labelsep=colon]{caption}
\usepackage[colorlinks=true,urlcolor=blue, citecolor=blue,linkcolor=blue,linktocpage,pdfpagelabels, bookmarksnumbered,bookmarksopen]{hyperref}
\usepackage[hyperpageref]{backref} 
\numberwithin{equation}{section}
\newtheorem{theorem}{Theorem}[section]
\newtheorem{lemma}[theorem]{Lemma}

\newtheorem{prop}[theorem]{Proposition}
\theoremstyle{definition}
\newtheorem{example}{Example}

\newtheorem{remark}[theorem]{Remark}
\theoremstyle{remark}

\newif\ifmarked
\markedfalse
\definecolor{revcolor}{rgb}{0.80,0,0}

\newcommand{\ds}{\displaystyle}

\newcommand{\R}{\mathbb{R}}

\newcommand{\de}{\partial}
\newcommand{\eps}{\varepsilon}

\newcommand{\Om}{\Omega}

\newcommand{\n}{\nabla}
\def\XXint#1#2#3{{\setbox0=\hbox{$#1{#2#3}{\int}$}
   \vcenter{\hbox{$#2#3$}}\kern-.5\wd0}}

\renewcommand\footnotemark{}

\numberwithin{equation}{section}

\title{Weighted Robin eigenvalue problems and nonlinear elliptic equations with general growth in the gradient}

\author{Francesco Della Pietra
\thanks{
F. Della Pietra: Dipartimento di Matematica e Applicazioni ``R. Caccioppoli'', Universit\`a degli studi di Napoli Federico II, Via Cintia, Complesso Universitario Monte S. Angelo, 80126 Napoli, Italy.
\href{mailto:f.dellapietra@unina.it}{\nolinkurl{f.dellapietra@unina.it}}
}
\and
 Giuseppina di Blasio
\thanks{G. di Blasio: Dipartimento di Matematica e Fisica, Università degli Studi della Campania “Luigi Vanvitelli”, viale Lincoln 5,81100 Caserta, Italy. Email: \href{mailto:giuseppina.diblasio@unicampania.it}{\nolinkurl{giuseppina.diblasio@unicampania.it}}
}
\and
Giuseppe Riey\thanks{G. Riey:
Dipartimento di Matematica e Informatica, Università della Calabria, Ponte Pietro Bucci 31B, Cosenza, Rende 87036, Italy. Email: \href{mailto:giuseppe.riey@unical.it}{\nolinkurl{giuseppe.riey@unical.it}}
}
}
\date{\today}

\begin{document}
\maketitle

\begin{abstract}
We prove an existence result for Robin boundary value problems modeled on
 \begin{equation*}
\begin{cases}
    \Delta u + |\nabla u|^2 + \lambda f(x) = 0 & \text{in } \Omega \\
    \frac{\partial u}{\partial \nu} + \beta u = 0 & \text{on } \partial\Omega
\end{cases}
\end{equation*}
where $\Omega$ is a bounded, sufficiently smooth open set in $\R^N$, $f(x)$ belongs to the Marcinkiewicz space $M^{\nicefrac N2}$ and {$\beta>0$}, under a smallness assumption on the datum $\lambda$. In order to study such problem, we will show several properties of the weighted, singular Robin eigenvalue problem
\[
\lambda_{1,f,\gamma}(\Omega)= \inf_{\psi\in H^{1},\;\int_{\Omega}f\psi^{2}=1}\left\{\int_{\Omega}|\nabla \psi|^{2}dx+\gamma\int_{\de\Omega}\psi^{2}\right\}.
\]
\end{abstract}

\begin{center}
\begin{minipage}{12cm}
\small
\tableofcontents
\end{minipage}
\end{center}

\section{Introduction}

In this paper we are concerned with the existence of solutions to a class of nonlinear elliptic boundary value problems with Robin boundary conditions. To introduce the subject, let us consider the prototype problem
\begin{equation}
\label{problemmainintro}
\begin{cases}
    -\Delta u = \sigma_{0}|\nabla u|^2 + \lambda f(x) & \text{in } \Omega \\
    \frac{\partial u}{\partial \nu} + \beta u = 0 & \text{on } \partial\Omega
\end{cases}
\end{equation}
where $\Omega$ is a bounded Lipschitz domain in $\R^N$ ($N\ge 3$), $\nu$ is the outer unit normal to $\partial\Omega$, $\lambda,\sigma_{0},\beta > 0$ are constants, and $f(x)$ is a nonnegative function belonging to the Marcinkiewicz space $M^{\nicefrac{N}{2}}(\Omega)$.

Problems with quadratic growth in the gradient, often referred to as having natural (or critical) growth, have been extensively studied, particularly in the case of Dirichlet boundary conditions, as well as its generalization to $p-$Laplace operator, or with subcritical growth. The bibliography is vast; we recall, for example, \cite{abddaper,afm,bmp92, dp,dpdb,dpg4,dpper,fm98,fm00,femu13,gmp12,hmv99, kazkra89,oliva,tr03}. The critical nature of the space $L^{\nicefrac{N}{2}}$ (or $M^{\nicefrac{N}{2}}$) for the source term $f$ is well-known in this context. Moreover, it is also well known that, in the Dirichlet case ($\beta=+\infty$), in order to get a $H^{1}$ solution to \eqref{problemmainintro} a smallness hypothesis on $\lambda$ is needed.
In particular, when $f\not\equiv 0$, this happens if $\lambda$ is smaller than the first weighted eigenvalue $\lambda_{1,f}(\Omega)$ of $-\Delta$, namely
\begin{equation}
\label{lambda1f}
\lambda_{1,f}(\Omega)=\inf_{\psi \in H^1_0(\Omega)\setminus\{0\}} \frac{\ds\int_\Omega |\nabla \psi|^2dx}{\ds\int_{\Omega}f(x)\psi^2dx}.
\end{equation}
This variational problem includes as special cases the standard first Dirichlet eigenvalue (when $f \equiv 1$) and the Hardy constant (when $f(x) = |x|^{-2}$ {and $0\in\Omega$}, $\lambda_{1,f}=\frac{(N-2)^{2}}{4}$).

The main novelty of this work lies in the treatment of the Robin boundary condition, combined with a source term in the Lorentz space $M^{\nicefrac{N}{2}}$, as for example the Hardy weight. While Dirichlet problems are naturally associated with the functional space $H^1_0(\Omega)$, the Robin condition requires working in the full space $H^1(\Omega)$, which introduces specific difficulties related to the trace terms. 
As a matter of fact, if we formally perform the classical change of variable $v=e^{u}$, problem \eqref{problemmainintro} becomes
\[
\begin{cases}
\Delta v + \lambda f v = 0 & \text{in } \Omega \\
\frac{\partial v}{\partial \nu} + \beta v \ln v = 0 & \text{on } \partial\Omega
\end{cases}
\]
The transformation linearizes the PDE but introduces a nonlinearity of logarithmic type into the Robin boundary condition.

In the Robin case, the existence of a solution to \eqref{problemmainintro} is deeply connected to the properties of a related weighted eigenvalue problem. Specifically, we investigate the first eigenvalue of the Laplacian with Robin boundary conditions and a singular weight $f$. We consider
\begin{equation}
\label{lambda1fgamma}
\lambda_{1,f,\gamma}(\Omega)=\inf_{\psi\in H^{1}(\Omega)\setminus\{0\}} R_{\gamma,f}[\psi]
\end{equation}
where
\[
R_{\gamma,f}[\psi] =  \frac{\ds\int_\Omega |\nabla \psi|^2dx+\gamma \int_{\de \Omega}\psi^{2}d\sigma}{\ds\int_{\Omega}f(x)\psi^2dx}.
\]
and $\gamma \in \mathbb{R}$. This variational problem generalizes the classical weighted Dirichlet eigenvalue \eqref{lambda1f} (obtained for $\gamma \to +\infty$) and includes as special cases the standard Robin eigenvalue (when $f \equiv 1$) and variants of the Hardy constant (when $f(x) = |x|^{-2}$). The case of the Hardy potential has been  studied by Chabrowski et al. in \cite{cpr, chab} and by Adimurthi et al. in \cite{adi, adiest}, who established Hardy-Sobolev type inequalities. Our hypotheses will cover a wide family of weights $f$. We will only impose a condition on $f$ to avoid singularities close to the boundary (see Section {\ref{sec:cc}}).

Then, the main existence result for problems like \eqref{problemmainintro} (Theorem \ref{thm:main_existence}) states that if the parameter $\lambda$ is sufficiently small, specifically
\[
\lambda < \frac{\lambda_{1,f}(\Omega)}{\sigma_0},
\]
then there exists at least one weak solution $u \in H^1(\Omega)$ to the problem. Furthermore, this solution possesses a regularity property typical of problems with quadratic growth: {$e^{\sigma_0 |u|} \in H^1(\Omega)$}. Moreover, it is interesting to observe that the smallness hypothesis for the existence does not depend on $\beta$, but only on the Dirichlet eigenvalue ($\beta=+\infty$).

To achieve this, we first perform a detailed analysis of the eigenvalue problem \eqref{lambda1fgamma}. We prove the existence of minimizers using a Concentration-Compactness argument to handle the possible lack of compactness due to the singular weight $f \in M^{\nicefrac{N}{2}}$, adapting the techniques of P.L. Lions \cite{lions,lions2} and Smets \cite{smets}. We also study the qualitative properties of $\lambda_{1,f,\gamma}(\Omega)$ as a function of $\gamma$, showing its continuity, monotonicity, and asymptotic behavior towards the Dirichlet eigenvalue $\lambda_{1,f}(\Omega)$.
The bibliography on this topic is quite large, too; we refer, for example, to \cite{smets, sw}, for the study of some properties of weighted eigenvalue properties in the Dirichlet case.

The proof of the existence theorem relies on a priori estimates obtained by using exponential test functions (a technique going back to {the works of Boccardo, Murat and Puel, see for instance \cite{bmp84,bmp92}}), adapted to handle the boundary terms and the arbitrary sign of the solution. The convergence of the approximate solutions is then established using truncation methods and pointwise convergence of gradients, as developed in \cite{BeBoMu, fm00, femu13}.

In summary, the aim {of the paper} is twofold: the study of a Robin eigenvalue problem with a very general weight, and the existence of solutions of equations modeled on \eqref{problemmainintro}. It is organized as follows: in Section 2 we introduce the functional setting and some preliminary results, including a concentration-compactness lemma. Section 3 is devoted to the study of the weighted Robin eigenvalue problem \eqref{lambda1fgamma}. Finally, in Section 4 we state and prove the existence result for the general nonlinear problem \eqref{generalrobinpb}.

\section{Preliminaries}

\subsection{Functional setting}
We recall that for $p > 0 $, the {Marcinkiewicz space} $M^p(\Omega)$, also known as the weak $L^p$ space, is the set of all measurable functions $f$ on $\Omega$ such that there exists a constant $C>0$ for which the distribution function of $f$, $\mu_{f}(t):=|\{x\colon |f(x)|>t\}|$, satisfies
\[
\mu_{f}(t)\le \frac{C}{t^{p}},\quad \forall t> 0.
\]
It is well-known that, when $|\Omega|<+\infty$, then $L^{q}(\Omega) \subset M^{q}(\Omega)\subset L^{q-\eps}(\Omega)$, for any $\eps>0$.

The  space $M^{q}$ is a particular case of the so called Lorentz spaces. A measurable function $f:\Omega\rightarrow \R$ belongs to the Lorentz space $L^{p,q}(\Omega)$, $1 < p< +\infty$, if
\begin{equation*}
  \|f\|_{p,q,\Omega}=\left\{
    \begin{array}{ll}
      \displaystyle\left\{\int_0^{|\Omega|} \left[t^{1/p} f^{*}(t)\right]^q
        \frac{dt}{t}\right\}^{1/q},&1\le q<+\infty,\\
      \displaystyle\sup_{0<t< |\Omega|} t^{1/p} f^{*}(t),&q=+\infty,
    \end{array}
  \right.
\end{equation*}
is finite, where $f^{*}(s)$, $s\in[0,|\Omega|]$, is the decreasing rearrangement of $f$, that is the distribution function of $\mu_{f}$.
In general $\|f\|_{{p,q,\Omega}}$ is not a norm. As matter of fact, it is
possible to introduce a {norm} in $L^{p,q}$ defining
\[
\|f\|_{(p,q)}=\|f^{**}\|_{p,q,\Omega},
\]
with $f^{**}(t)=t^{-1}\int_0^t f^*(\sigma)\,d\sigma$.
{W}e have that for $1<p< +\infty$ and $1\le q \le +\infty$,
\[
\|f\|_{p,q,\Omega}\le\|f\|_{(p,q)}\le \frac{p}{p-1}\|f\|_{p,q,\Omega}.
\]

For $1\le p,q \le +\infty$, it holds that
\begin{equation}
\label{holdlor}
\|fg\|_{L^{1}(\Omega)} \le \|f\|_{p,q,\Omega}\|g\|_{p',q',\Omega},
\end{equation}
where $p',q'$ are the H\"older conjugates of $p,q$ respectively.
More generally, the $L^{p,q}$ spaces are related in the
following way:
\[
    L^r\subset L^{p,1}\subset L^{p,q}\subset L^{p,p}=L^p\subset
    L^{p,r}\subset L^{p,\infty}=M^{p}\subset L^q,
\]
for $1 < q < p < r < +\infty$.  The Lorentz spaces allow to get a refinement of the Sobolev inequality. Indeed, it holds that
\begin{equation*}
\|f\|_{p^{*},p,\R^{N}} \le S_{N,p}\|f\|_{W^{1,p}(\R^{N})}
\end{equation*}
for $1\le p<N$ (see \cite{alv}). Then, if $\Omega$ is a bounded Lipschitz domain, using classical extension theorems we can conclude that there exists a positive constant $C$ such that
\begin{equation}
\label{sobolor}
\|f\|_{p^{*},p,\Omega} \le C\|f\|_{W^{1,p}(\Omega)}
\end{equation}
for any $f\in W^{1,p}(\Omega)$.

In the following, the Proposition below will be useful.
\begin{prop}
\label{propequiv}
Let $\Omega \subset \mathbb{R}^N$ be a bounded, connected, Lipschitz domain, and let $f \in L^1(\Omega)$ {be a nonnegative function, $f\not\equiv 0$}.
Then there exists a positive constant $C$ such that for every $u \in H^1(\Omega)$,
\begin{equation*}
\label{propequiveq}
\|u\|_{H^1(\Omega)} \le C \left( \|\nabla u\|_{L^2(\Omega)} + \|u\|_{L^2(\Omega,f)} \right),
\end{equation*}
where $\|u\|_{L^2(\Omega,f)}=\left( \int_\Omega f u^2 dx \right)^{1/2}$.
\end{prop}

\begin{proof}
We prove the inequality by contradiction. Therefore, for any $n \ge 1$, there exists $u_n \in H^1(\Omega)$ such that
\[
\|u_n\|_{H^1(\Omega)} > n \left( \|\nabla u_n\|_{L^2(\Omega)} + \left(\int_\Omega f u_n^2 dx\right)^{1/2} \right).
\]
Let us define $v_n = \frac{u_n}{\|u_n\|_{H^1(\Omega)}}$. By construction, the sequence $\{v_n\}$ satisfies
\begin{enumerate}[(i)]
    \item $\|v_n\|_{H^1(\Omega)} = 1$ for all $n \in \mathbb{N}$,
    \item $\|\nabla v_n\|_{L^2(\Omega)} + \left(\int_\Omega f v_n^2 dx\right)^{1/2} < \frac{1}{n}$.
\end{enumerate}
From  (ii), it follows immediately that as $n \to \infty$:
\begin{equation} \label{eq:convergences_en}
    \|\nabla v_n\|_{L^2(\Omega)} \to 0 \quad \text{and} \quad \int_\Omega f(x) v_n^2 dx \to 0.
\end{equation}
From  (i), the sequence $\{v_n\}$ is bounded in $H^1(\Omega)$. Therefore, up to a subsequence, $v_n \rightharpoonup v_0$ weakly in $H^1(\Omega)$, and $v_n \to v_0$ in $L^2(\Omega)$ and almost everywhere in $\Omega$. From \eqref{eq:convergences_en}, we must have $\nabla v_0 = 0$ and,
since $\Omega$ is a connected domain, this implies that $v_0 \equiv K$. Furthermore, combining $\|v_n\|_{H^1}^2 = \|\nabla v_n\|_{L^2}^2 + \|v_n\|_{L^2}^2 = 1$ and that $\|\nabla v_n\|_{L^2} \to 0$, it holds $\|v_n\|_{L^2} \to 1$. As $v_n \to v_0$ strongly in $L^2(\Omega)$, it follows that $\|v_0\|_{L^2} = 1,$ and $ v_{0}=K = \pm 1/\sqrt{|\Omega|}$.
Finally, by Fatou's Lemma and \eqref{eq:convergences_en}, it holds
\[
    \int_\Omega \liminf_{{n}\to\infty} (f(x) v_{n}(x)^2) \, dx \le \liminf_{{n}\to\infty} \int_\Omega f(x) v_{n}^2 \, dx=0.
\]
This yields:
\[
    \int_\Omega f(x) v_0^2 \, dx = 0.
\]
{This} is a contradiction, {since $\int_\Omega f v_0^2\,dx=K^2\int_\Omega f\,dx>0$, being $f\ge0$, $f\not\equiv0$ and $K\neq0$}. The proof is complete.
\end{proof}

\begin{prop}
Let $\Omega \subset \mathbb{R}^N$ be a bounded, {connected,} Lipschitz domain, and let $f \in L^1(\Omega)$ {be} such that $f > 0$ almost everywhere in $\Omega$. Then, for every $\eps > 0$, there exists a constant $C = C(\varepsilon) > 0$ such that for all $u \in H^1(\Omega)$, the following inequality holds:
\begin{equation}
\label{traceweighted}
\int_{\partial\Omega} u^2 \, d\sigma \le \eps \int_{\Omega} |\nabla u|^2 \, dx + C(\eps)\int_{\Omega} u^2 f \, dx
\end{equation}
\end{prop}

\begin{proof}
We proceed by contradiction. Then there exists an $\epsilon_0 > 0$ such that for all $m \in \mathbb{N}$, we can find a function $u_m \in H^1(\Omega)$ (with $u_m \not\equiv 0$) satisfying
\[
\varepsilon_0 \int_{\Omega} |\nabla u_m|^2 \, dx + m \int_{\Omega} u_m^2 f \, dx < \int_{\partial\Omega} u_m^2 \, d\sigma =1,\quad \text{for all } m
\]
Hence, by {Proposition \ref{propequiv}}, $u_m$ is bounded in $H^1(\Omega)$, and then
\[
u_m \rightharpoonup u \quad \text{weakly in } H^1(\Omega),\quad u_m \to u \quad \text{strongly in } L^2(\Omega)\text{ and in }L^2(\partial\Omega).
\]
Hence
\begin{equation}
\label{passtraceineq}
\int_{\partial\Omega} u^2 \, d\sigma = 1.
\end{equation}
Moreover, by Fatou {Lemma, since $m\int_\Omega u_m^2f\,dx<1$,} we get
\[
\int_{\Omega} u^2 f \, dx \le {\liminf_{m \to \infty}} \int_{\Omega} u_m^2 f \, dx = 0,
\]
and thus $u =0$ a.e. in $\Omega${, since $f>0$ a.e.}, and so its trace, contradicting \eqref{passtraceineq}.
\end{proof}
\subsection{A concentration-compactness Lemma}\label{sec:cc}
The analysis of minimizing sequences for \eqref{lambda1fgamma} is challenging because the functional $u \mapsto \int_\Omega f u^2 dx$ is not weakly continuous on $H^1(\Omega)$ for a general weight $f \in M^{\nicefrac{N}{2}}(\Omega)$. To overcome this lack of compactness, we use the concentration-compactness Principle, in the spirit of the well-known results of P.L. Lions (\cite{lions,lions2}), and of the result proved by Smets (\cite{smets}) in the Dirichlet case.

To this aim, we need to assume a stronger assumption on $f$. For any $x \in \overline{\Omega}$ and $r>0$, we define{, extending $f$ by zero outside $\Omega$,}
\[
S_{r,f}^x := \lambda_{1,f}({B_{r}(x)})= \inf \left\{ \int_{{B_r(x)}} |\nabla \phi|^2 d{y} \colon \phi \in H_0^1({B_r(x)}), \int_{\Omega \cap B_r(x)} f({y})\phi^2 d{y} = 1 \right\},
\]
and we set for $x\in \overline\Omega$
 \[
 S_f^x := \lim_{r\to 0^+} S_{r,f}^x.
 \]
 The singularity set $\Sigma_f$ is
\[
\Sigma_f := \{ x \in \overline{\Omega} \mid S_f^x < +\infty \}.
\]
For an interior point $x \in \Omega$ and $r$ sufficiently small, $S_{r,f}^x$ is the first eigenvalue of the Dirichlet problem on $B_{r}(x)$ with weight $f$, $\lambda_{1,f}(B_{r}(x))$. {On the other hand, if $x\in\partial\Omega$, the admissible functions in the definition of $S^x_{r,f}$ are not required to vanish on $\partial\Omega\cap B_r(x)$: this is the relevant situation for sequences in $H^1(\Omega)$, which may concentrate at boundary points (see Lemma \ref{cclemma} and Remark \ref{rem:boundary}).} We will make the following two assumptions:
\begin{enumerate}
\item[$(H_{f,1})$] $\overline\Sigma_{f}$ is at most countable;
\item[$(H_{f,2})$] the singularity set $\Sigma_f$ is compactly contained in $\Omega$:
\[
\overline\Sigma_f \cap \partial\Omega = \emptyset.
\]
\end{enumerate}
The first condition {avoids} strong spikes on a dense subset of $\Omega$; the second {avoids} that singularities on the boundary of $\Omega$ appear.

\begin{remark}\label{2.3}
We emphasize that when  $f\in L^{\nicefrac N2}(\Omega)$, it holds that $\Sigma_{f}=\emptyset$.

Indeed in such a case, for $x\in\overline \Omega$
\[
S_{r,f}^{x} {=} \lambda_{1,f}(B_{r}(x)),\qquad \text{and } \lim_{r \to 0^+} \lambda_{1,f}(B_r) = +\infty
\]
{(recall that $f$ is extended by zero outside $\Omega$)}. Through a change of variables,
\[
\lambda_{1,f}(B_r) = \frac{1}{r^2} \lambda_{1,f_r}(B_1),
\]
where $f_r(z) = f({x+}rz)$, $z=\frac{y-x}{r}$ is the scaled weight function defined on $B_1$.
We look for a lower bound for $\lambda_{1,f_r}(B_1)$. For any test function $v \in H^1_0(B_1)$, using Hölder and Sobolev inequalities we get
\[
\int_{B_1} f_r v^2 dz \le S \|f_r\|_{L^{\nicefrac{N}{2}}(B_1)} \int_{B_1} |\nabla v|^2 dz
\]
where $S$ is the Sobolev constant.
Hence
\[
\frac{\displaystyle\int_{B_1} |\nabla v|^2 dz}{\displaystyle\int_{B_1} f_r v^2 dz} \ge \frac{1}{S \|f_r\|_{L^{\nicefrac{N}{2}}(B_1)}}.
\]
Passing to the infimum, we may conclude that
\[
\lambda_{1,f}(B_r) \ge \frac{1}{r^2} \cdot \frac{1}{S \|f_r\|_{L^{\nicefrac{N}{2}}(B_1)}}=\frac{1}{S \left(\displaystyle\int_{B_r} |f|^{\frac N2} dy\right)^{\frac2N}}.
\]
Being $f\in L^{\nicefrac N2}(\Omega)$, we conclude that
\[
\lim_{r \to 0^+} \lambda_{1,f}(B_r) = +\infty.
\]
\end{remark}
\begin{example} If $f=\frac{1}{|x|^{2}}$ and $0\in \Omega$, then $S^{0}_{f}=\frac{(N-2)^{2}}{4}$ is the Hardy constant, and $\Sigma_{f}=\{0\}$.
\end{example}

 We are ready to state and prove the concentration-compactness lemma.
\begin{lemma} \label{cclemma}
Let $\Omega \subset \mathbb{R}^N$ be a bounded, {Lipschitz,} open, and connected set, and let $f \in M^{\nicefrac{N}{2}}(\Omega)$ be a non-negative weight function, and $(H_{f,1})-(H_{f,2})$ hold.
Consider a bounded sequence $\{u_n\}$ in $H^1(\Omega)$ with $u_n \rightharpoonup u_0$ weakly in $H^1(\Omega)$. We may assume that there exist non-negative, finite Radon measures $\mu, \nu, \tilde{\mu}$ on $\Omega$
such that:
\begin{enumerate}
\item[a.] $|\nabla(u_n - u_0)|^2  \rightharpoonup \mu$ weakly in $\mathcal M^{+}(\Omega)$;
\item[b.] $f|u_n - u_0|^2  \rightharpoonup \nu$ weakly in $\mathcal M^{+}(\Omega)$;
\item[c.]  $|\nabla u_n|^2  \rightharpoonup \tilde{\mu}$ weakly in $\mathcal M^{+}(\Omega)$,
\end{enumerate}
where $\mathcal M^{+}(\Omega)$ is the cone of positive finite Radon measures over $\Omega$. Then, the following properties hold:
\begin{enumerate}[(1)]
    \item The measure $\nu$ is purely atomic and supported on $ \Sigma_f$. That is,
    \[ \nu = \sum_{i \in I} \nu_i \delta_{x_i}, \quad \text{with } x_i \in \Sigma_f \text{ and } \nu_i > 0, \]
    where $I$ is an at most countable set.
    \item 
    The measure $\mu$ is bounded from below by $\nu$, in the sense that
    \begin{equation}
    \label{ineqmu}
     \mu \ge \sum_{i \in I} \nu_i S_f^{x_i} \delta_{x_i},
    \end{equation}
   and the total energy measure $\tilde{\mu}$ satisfies:
      \begin{equation}
    \label{ineqmutilde}
    \tilde{\mu} \ge |\nabla u_0|^2 + \sum_{i \in I} \nu_i S_f^{x_i} \delta_{x_i}.
    \end{equation}
    \item
        \label{3punto}
     {No mass is lost at the boundary, that is,
    \[
    \lim_{n\to\infty}\int_\Omega f|u_n-u_0|^2dx=\nu(\Omega),\qquad\liminf_{n\to\infty}\int_\Omega|\nabla(u_n-u_0)|^2dx\ge\mu(\Omega).
    \]}
\end{enumerate}
\end{lemma}

\begin{proof}
The proof adapts the argument by P.L. Lions (\cite{lions,lions2}) and Smets (\cite{smets}), adding the additional hypothesis $\overline\Sigma_f \cap \partial\Omega = \emptyset$ in order to manage a sequence in $H^1(\Omega)$. In what follows, we set $v_n = u_n - u_0$. {Moreover, since $\Omega$ is a Lipschitz domain, $v_{n}$ can be extended by means of a bounded linear extension operator to all $\R^{N}$; we still denote this extension by $v_{n}$. It holds that $v_n\rightharpoonup0$ weakly in $H^1(\R^N)$ and $v_n\to0$ strongly in $L^2_{\rm loc}(\R^N)$.}

We first claim that {for any test function $\phi\in C^\infty_0(\R^N)$ with $\text{supp}(\phi)\cap\overline\Sigma_f=\emptyset$ it holds
\begin{equation*}
\lim_{n\to\infty}\int_\Omega f v_n^2\phi\,dx=0. \tag{$\ast$}
\end{equation*}
In particular, choosing $\phi\in C^\infty_0(\Omega)$, we get $\nu(\phi)=\int_\Omega\phi\,d\nu=0$, that is,} {$\nu$ is supported on $\overline\Sigma_f$}. Let $\phi$ be such a test function. We may assume that $\phi\ge 0$ and $\sqrt{\phi}$ is smooth. For any $x\in \text{supp}(\phi){\cap\overline\Omega}$, since $x \notin \overline\Sigma_f$, by definition $S_f^{x} = +\infty$. This means that for any large number $M$, we can choose the radius $r$ of each ball $B_{r}$ small enough such that $S_{r,f}^{x} > M$. We can then extract a finite open sub-covering $\{B_{r_{i}}(x_{i})\}_{i=1,\ldots,m}$ {of the compact set $\text{supp}(\phi)\cap\overline\Omega$}.
Now, let $\{\xi_i\}_{i=1}^m$ be a smooth partition of unity subordinate to this cover. Consider the test functions $w_{n,i} = v_n\cdot (\phi\xi_i)^{1/2}$. Since $\phi\xi_i$ has compact support inside the ball $B_i$, the function $w_{n,i}$ belongs to $H_0^1(B_i)$. It is therefore an admissible test function for the infimum defining $S_{r_i,f}^{x_i}${, also when $x_i\in\partial\Omega$}. Applying the definition, we have:
\[
\int_{{B_i}} |\nabla w_{n,i}|^2 dx \ge S_{r_i,f}^{x_i} \int_{{\Omega\cap B_i}} f w_{n,i}^2 dx > M \int_\Omega f v_n^2 \phi\xi_i dx.
\]
Then{, summing over $i$,
\begin{multline}
\label{passaggio}
\int_{\Omega} f v_{n}^{2} \phi\,dx=\sum_{i=1}^m \int_{\Omega} f v_{n}^{2} \phi \xi_i\,dx  < \frac{1}{M} \sum_{i=1}^m \int_{\R^N} |\nabla w_{n,i}|^2 dx \\
\le \frac{1+\epsilon}{M}\int_{\R^N}\phi|\nabla v_n|^2dx+\frac{C_\epsilon}{M}\sum_{i=1}^m\int_{\R^N}|v_n|^2\left|\nabla((\phi\xi_i)^{1/2})\right|^2dx.
\end{multline}}

The {last inequality in} \eqref{passaggio}
follows by using the elementary inequality $| |a+b|^2 - |a|^2 | \le \epsilon|a|^2 + C_\epsilon|b|^2$:
    \begin{equation}
    \label{passaggio2}
    \left| \int_{{\R^N}} |\nabla w_{n,i}|^2 - \int_{{\R^N}} \phi\xi_i |\nabla {v_n}|^2 \right| \le \epsilon \int_{{\R^N}} \phi\xi_i |\nabla {v_n}|^2 + C_\epsilon \int_{{\R^N}} |{v_n}|^2 \left|\nabla((\phi\xi_i)^{1/2})\right|^2,
    \end{equation}
    {and $\sum_i\xi_i\le1$. The last term in \eqref{passaggio} tends to zero as $n \to \infty$, because $v_n \to 0$ strongly in $L^2_{\rm loc}(\R^N)$, while $\int_{\R^N}\phi|\nabla v_n|^2dx\le C\|\phi\|_\infty$, with $C$ independent of $n$ and $M$. Hence
\[
\limsup_{n\to\infty}\int_\Omega fv_n^2\phi\,dx\le\frac{(1+\epsilon)\,C\|\phi\|_\infty}{M}.
\]}
As $M$ can be arbitrarily large, we conclude {that ($\ast$) holds, proving the claim}.

Now, let $x \in \Sigma_f$; for each $j \in \mathbb{N}$ there exists $r_j>0$ such that $S_{r_j, f}^x>S_f^x-\frac 1j$.
We can choose a sequence of smooth cutoff functions $\{\psi_j\}$ that are supported in balls $B_{r_j}(x)$ with $r_j \to 0$ as $j \rightarrow \infty$, and $\psi_j \equiv 1$ near $x$. We have, arguing in particular as in \eqref{passaggio2},
\begin{multline*}
\mu(\{x\})  =\lim _{j \rightarrow \infty} \mu\left(\psi_j^2\right)  =\lim _{j \rightarrow \infty} \overline{\lim _{n \rightarrow \infty}} \int_{\Omega}\left|\nabla\left(v_{n} \psi_j\right)\right|^2 \\
 \geqslant \lim _{j \rightarrow \infty}\left\{\left(S_f^x-\frac1j\right) \overline{\lim _{n \rightarrow \infty}} \int_{\Omega} f v_{n}^2 \psi_j^2\right\} =\lim _{j \rightarrow \infty}\left(S_f^x-\frac1j\right) \nu\left(\psi_j^2\right)  =S_f^x \nu(\{x\}).
\end{multline*}
Similarly, if $x \in \overline{\Sigma_f} \backslash \Sigma_f$, it holds that $\mu(\{x\})>j \nu(\{x\})$ for each $j$ so that $\nu(\{x\})=0$. This proves that there are no Dirac masses on $\overline\Sigma_f\setminus \Sigma_{f}$. Moreover, $\mu(\{x\})=\tilde{\mu}(\{x\})$. Indeed
\begin{multline*}
\left|\int_{\Omega}\left| \nabla v_{n}\right|^2 \psi_j^2 dx-\int_{\Omega}\left|\nabla u_n\right|^2 \psi_j^2dx \right| \\
 \leqslant \varepsilon \int_{\Omega}\left|\nabla v_{n}\right|^2 \psi_j^2+c(\varepsilon) \int_{\Omega}|\nabla u{_0}|^2 \psi_j^2 \leqslant \varepsilon K+c(\varepsilon) o(1), \quad j \rightarrow \infty.
\end{multline*}
Now, if $\psi \in C^{\infty}_0\left({\Omega}\right)$ is nonnegative, by the weak sequential lower semi-continuity of $v\in L^2(\Omega){^N} \mapsto \int_{\Omega} \psi|v|^2$, we get $\tilde\mu \geqslant|\nabla u_{0}|^2$. One can conclude by the orthogonality of $|\nabla u_{0}|^2$ with Dirac masses.

{Finally, we prove 3. We follow the argument of \cite[Lemma 2.1, (3)]{smets}. By $(H_{f,2})$, $\overline\Sigma_f$ is a compact subset of $\Omega$; hence there exists $\phi_0\in C^\infty_0(\R^N)$, with $0\le\phi_0\le1$, $\phi_0\equiv1$ in a neighborhood of $\partial\Omega$ and $\text{supp}(\phi_0)\cap\overline\Sigma_f=\emptyset$. Since $1-\phi_0$ has compact support in $\Omega$, and $\nu$ is supported on $\overline\Sigma_f$, where $\phi_0=0$, by ($\ast$) we get
\[
\lim_{n\to\infty}\int_\Omega fv_n^2dx=\lim_{n\to\infty}\int_\Omega fv_n^2(1-\phi_0)\,dx+\lim_{n\to\infty}\int_\Omega fv_n^2\phi_0\,dx=\nu(1-\phi_0)=\nu(\Omega).
\]
The inequality for $\mu$ follows from the lower semicontinuity of the total mass with respect to the weak convergence of measures.}
This completes the proof.
\end{proof}

\section{Some properties of the first weighted Robin eigenvalue}
In this section we list and prove several properties of $\lambda_{1,f,\gamma}(\Omega)$. For the sake of completeness, we consider both the cases $\gamma>0$ and $\gamma<0$.
\begin{prop}[Simplicity of the first eigenvalue]
\label{simplicity}
Let $\Omega \subset \mathbb{R}^N$ be a bounded, connected open set with Lipschitz boundary.
Let $f \in M^{\nicefrac{N}{2}}(\Omega)$ be a weight function such that $f > 0$ a.e. in $\Omega$. If $\lambda_{1,f,\gamma}$ is attained by a minimizer, then this is unique up to a multiplicative constant.
\end{prop}

\begin{proof}

Let $u \in H^1(\Omega)$ be a minimizer for $\lambda_{1,f,\gamma}$. Then $|u|$ is still a minimizer, so we may assume $u\ge 0$. By the maximum principle (see \cite[Cor. 3]{brezponce}), it holds that $u>0$ a.e. in $\Omega$. We now show that any two positive minimizers, $u_1$ and $u_2$, must be proportional.

Let $u_1 > 0$ and $u_2 > 0$ be two normalized minimizers of $\lambda = \lambda_{1,f,\gamma}$, $\int_\Omega f u_1^2 dx = \int_\Omega f u_2^2 dx = 1$.
 For any test function $\phi \in H^1(\Omega)$:
\begin{align}
\int_\Omega \nabla u_1 \cdot \nabla\phi \,dx + \gamma\int_{\partial\Omega} u_1 \phi \,d\sigma &= \lambda \int_\Omega f u_1 \phi \,dx \label{eq:EL1} \\
\int_\Omega \nabla u_2 \cdot \nabla\phi \,dx + \gamma\int_{\partial\Omega} u_2 \phi \,d\sigma &= \lambda \int_\Omega f u_2 \phi \,dx \label{eq:EL2}
\end{align}
Subtracting \eqref{eq:EL2} from \eqref{eq:EL1}, the function $w=u_1-u_2$ satisfies
\[
\int_\Omega\nabla w\cdot\nabla\phi\,dx+\gamma\int_{\partial\Omega}w\phi\,d\sigma=\lambda\int_\Omega fw\phi\,dx\qquad\text{for every }\phi\in H^1(\Omega).
\]
Choosing $\phi=w^+$ and $\phi=w^-$, and observing that $\nabla w^+\cdot\nabla w^-=0$ and $w^+w^-=0$ a.e. in $\Omega$, and $w^+w^-=0$ $\sigma$-a.e. on $\partial\Omega$, we get
\[
\int_\Omega|\nabla w^\pm|^2dx+\gamma\int_{\partial\Omega}(w^\pm)^2d\sigma=\lambda\int_\Omega f(w^\pm)^2dx.
\]
Hence, if $w^+\not\equiv0$, then $w^+$ is a nonnegative minimizer for $\lambda_{1,f,\gamma}$ (recall that $f>0$ a.e.), and, by the maximum principle as above, $w^+>0$ a.e. in $\Omega$, so that $w^-\equiv0$; in the same way, if $w^-\not\equiv0$, then $w^+\equiv0$. Therefore, if $w\not\equiv0$, either $u_1>u_2$ a.e. in $\Omega$, or $u_2>u_1$ a.e. in $\Omega$. This is impossible, since, by the normalization,
\[
\int_\Omega f(u_1-u_2)(u_1+u_2)\,dx=\int_\Omega fu_1^2dx-\int_\Omega fu_2^2dx=0,
\]
while $f>0$ and $u_1+u_2>0$ a.e. in $\Omega$. Hence $u_1=u_2$. Finally, the same argument, applied to a minimizer $u$ in place of $w$, shows that every minimizer has constant sign, and the proof is completed.
\end{proof}

\begin{prop} \label{prop:eigenvalue_properties}
Let $\Omega \subset \mathbb{R}^N$ be a bounded, connected open set with a Lipschitz boundary. Let $f \in M^{\nicefrac{N}{2}}(\Omega)$ be a weight function such that $f > 0$ a.e. in $\Omega$. Assume that the structural hypotheses $(H_{f,1})-(H_{f,2})$ hold. Then the first Robin eigenvalue $\lambda_{1,f,\gamma}(\Omega)$, $\gamma\in \R$ has the following properties:
\begin{enumerate}[(i)]
    \item
    \begin{enumerate}[(1)]
    \item For any $\gamma \in \R$, $\lambda_{1,f,\gamma}$ has the sign of $\gamma$; moreover,
    \begin{equation}
        \label{ineq1}
           -\infty< \lambda_{1,f,\gamma}(\Omega) \le \lambda_{1,f}(\Omega)<+\infty.
    \end{equation}
    \item the following limits hold:
    \begin{equation}
    \label{liminfty}
    \lim_{\gamma \to +\infty} \lambda_{1,f,\gamma}(\Omega) = \lambda_{1,f}(\Omega),
    \end{equation}
    and
    \begin{equation}
    \label{liminfty2}
        \lim_{\gamma \to -\infty} \lambda_{1,f,\gamma}(\Omega) = -\infty.
    \end{equation}
    \end{enumerate}
    \item The infimum $\lambda_{1,f,\gamma}(\Omega)$ is attained, provided
    \begin{equation}
    \label{condg}
    \lambda_{1,f,\gamma}(\Omega) < \lambda_{1,f}(\Omega).
    \end{equation}
    In particular, if $f \in L^{\nicefrac{N}{2}}(\Omega)$, this condition is always satisfied for $\gamma > 0$, and thus a minimizer always exists. 
      {Moreover, every minimizer for $\lambda_{1,f,\gamma}(\Omega)$ has nonzero trace on $\partial\Omega$. As a consequence, if a minimizer exists for some $\gamma$, then \eqref{condg} holds for every $\gamma'<\gamma$.}
    
    \item The function $\lambda(\gamma) := \lambda_{1,f,\gamma}(\Omega)$, defined on $\R$, is:
    \begin{enumerate}
        \item non-decreasing and continuous;
        \item differentiable for any $\gamma$ {where \eqref{condg} holds}. Its derivative is given by
        \[
\lambda'(\gamma) = \int_{\partial\Omega} u_{\gamma}^{2} d\sigma,
        \]
        where $u_{\gamma}$ is the unique nonnegative minimizer normalized by $\int_\Omega f u_\gamma^2 dx = 1$.
        \item If $f \in L^{\nicefrac{N}{2}}(\Omega)$, then {\eqref{condg} holds for every $\gamma\in\R$}, and the map is strictly increasing {on $\R$}.
    \end{enumerate}
\end{enumerate}
\end{prop}

\begin{proof}
We split the proof in several steps.
\medskip

\textit{\underline{Step 1}: For any $\gamma\in \R$, $\lambda_{1,f,\gamma}(\Omega)$ is finite}. If $\gamma\ge 0$, this is obvious by the definition. If $\gamma<0$, using the trace inequality \eqref{traceweighted}, we have, for any $\psi\in H^{1}$,
\[
\int_{\Omega}|\nabla \psi|^{2}dx+\gamma \int_{\de\Omega} \psi^{2}d\sigma \ge (1-\eps |\gamma|)\int_{\Omega}|\nabla \psi|^{2}dx -|\gamma|C(\eps)\int_{\Omega}\psi^{2}fdx.
\]
Hence, if $\eps< 1/|\gamma|$, we get that $R_{\gamma}[\psi]$ is bounded from below in $H^{1}$, which implies that $\lambda_{1,f,\gamma}(\Omega)$ is finite.

In all the following steps of the proof, we will always denote by $\{u_n\}_{n\in\mathbb{N}} \subset H^1(\Omega)$ a normalized minimizing sequence for $\lambda_{1,f,\gamma}(\Omega)$, in the sense that
\begin{equation}
\label{normmin}
   \lambda_{1,f,\gamma}(\Omega)=   \lim_{n\to\infty} \left( \int_\Omega |\nabla u_n|^2 dx + \gamma \int_{\partial\Omega} u_n^2 d\sigma \right),\qquad \int_\Omega f(x) u_n^2 dx = 1.
\end{equation}
As $\lambda_{1,f,\gamma}$ is finite, $u_{n}$ will be bounded in $H^{1}(\Omega)$ {(by Proposition \ref{propequiv} and, if $\gamma<0$, by the previous inequality)}, and then weakly convergent in $H^{1}$ and strongly in $L^{2}(\Omega)$ and $L^{2}(\de\Omega)$, respectively, to a function $u_{0}\in H^{1}(\Omega)$.
\medskip

\textit{\underline{Step 2:} proof of i)(1)}.
If $\gamma=0$, then $\lambda_{1,f,\gamma}=0$, and if $\gamma<0$, it holds $\lambda_{1,f,\gamma}\le \gamma\frac{P(\Omega)}{{\|f\|_{L^1(\Omega)}}}< 0$ by taking constant test functions. So let us suppose $\gamma>0$. In order to prove that $\lambda_{1,f,\gamma}(\Omega)>0$, we argue by contradiction. Assume that $\lambda_{1,f,\gamma}(\Omega) = 0$. Since $\gamma > 0$ we have by \eqref{normmin}:
\begin{equation} \label{eq:grad_trace_zero}
    \lim_{n\to\infty} \|\nabla u_n\|_{L^2(\Omega)} = 0 \quad \text{and} \quad \lim_{n\to\infty} \|u_n\|_{L^2(\partial\Omega)} = 0.
\end{equation}

This implies that $\nabla u_n \to 0$ strongly in $L^2(\Omega)^N$, then $\nabla u_0 = 0$, which implies that $u_0$ is a constant, say $K$, since $\Omega$ is connected. By the compactness of the trace operator, the weak convergence $u_n \rightharpoonup u_0$ in $H^1(\Omega)$ implies strong convergence of the traces $u_n \to u_0$ in $L^2(\partial\Omega)${; moreover, $u_n\to u_0$ strongly in $L^2(\Omega)$}. From \eqref{eq:grad_trace_zero}, we get $u_0 \equiv 0${, and then $\|u_n\|_{H^1(\Omega)}\to0$}.

On the other hand, by the H\"older inequality \eqref{holdlor} and the Sobolev inequality \eqref{sobolor} it holds that
\begin{equation*}
1=\int_{\Omega} fu_{n}^{2}dx \le \|f\|_{\nicefrac N2,\infty,\Omega}\|u_{n}\|^{2}_{\frac{2N}{N-2},2,\Omega} \le C\|f\|_{\nicefrac N2,\infty,\Omega}\|u_{n}\|^{2}_{H^{1}(\Omega)}
\end{equation*}
and this is a contradiction with $\|u_n\|_{H^1(\Omega)}\to 0$. This means that $\lambda_{1,f,\gamma}(\Omega)$ must be positive.
\medskip

\textit{\underline{Step 3:} proof of (ii)}. Let us prove the existence of the minimum under the assumptions $(H_{f,1})-(H_{f,2})$ and the condition $\lambda_{1,f,\gamma}<\lambda_{1,f}$. Let $u_{n}$ be a minimizing normalized sequence converging weakly in $H^{1}(\Omega)$ to $u_{0}$ and strongly in $L^{2}(\Omega)$ and $L^{2}(\de\Omega)$, as in the previous step. We aim to apply the concentration-compactness Lemma \ref{cclemma}. To this end, using the notation of Section {\ref{sec:cc}}, by definition it holds
\[
\lambda_{1,f}(B_{r}(x)) = S^{x}_{r,f},
\]
for $x\in\Omega$ and $r$ sufficiently small. Then, by inclusion,
\[
\lambda_{1,f}(\Omega) \le \lambda_{1,f}(B_{r}(x)) \le S_{f}^{x}.
\]
Hence from \eqref{ineqmu} we get
\begin{equation}
\label{eqqq}
\mu \ge \lambda_{1,f}(\Omega) \nu.
\end{equation}
\textit{Case $\gamma>0$.} First, we show that $u_{0}\not\equiv 0$. If by contradiction $u_{0}\equiv 0$ (and then $\tilde \mu=\mu$), using also \eqref{eqqq}{, \eqref{ineqmutilde} and Lemma \ref{cclemma} (3)}
\[
\lambda_{1,f,\gamma}(\Omega)=\lim_{n\to+\infty}\left[ \int_{\Omega}|\nabla u_{n}|^{2} dx+\gamma\int_{\de\Omega}u_{n}^{2}d\sigma \right] {\ge} \mu(\Omega) \ge \lambda_{1,f}(\Omega)\nu(\Omega) = \lambda_{1,f}(\Omega),
\]
being $f|u_{n}|^{2}=f|u_{n}-u_{0}|^{2}\rightharpoonup\nu$, and $\int_{\Omega}fu_{n}^{2}dx=1$ {(so that $\nu(\Omega)=1$, by Lemma \ref{cclemma} (3))} and $u_{n}\to 0$ strongly also in $L^{2}(\de\Omega)$. Due to $\lambda_{1,f,\gamma}<\lambda_{1,f}$, we have a contradiction.

Hence, the limit $u_{0}$ is non null. Let $v_n = u_n - u_0$; then $v_n \rightharpoonup 0$. {By the weak convergence in $H^1(\Omega)$, and since the bilinear form $(u,v)\mapsto\int_\Omega fuv\,dx$ is continuous on $H^1(\Omega)$ by \eqref{holdlor} and \eqref{sobolor},} we can decompose the norms:
\begin{align*}
\int_\Omega |\nabla u_n|^2 dx &= \int_\Omega |\nabla u_0|^2 dx + \int_\Omega |\nabla v_n|^2 dx + o(1), \\
1 = \int_\Omega f u_n^2 dx &= \int_\Omega f u_0^2 dx + \int_\Omega f v_n^2 dx + o(1).
\end{align*}
Then, by \eqref{eqqq} {and Lemma \ref{cclemma} (\ref{3punto})},
\begin{multline*}
\lambda_{1,f,\gamma} \ge \int_{\Omega} |\nabla u_{0}|^{2} dx+ {\mu(\Omega)} +\gamma\int_{\de\Omega}u_{0}^{2}d\sigma
\ge \int_{\Omega} |\nabla u_{0}|^{2} dx+ \lambda_{1,f}{\nu(\Omega)} +\gamma\int_{\de\Omega}u_{0}^{2}d\sigma \\
\ge \lambda_{1,f,\gamma} (1-{\nu(\Omega)}) + \lambda_{1,f}{\nu(\Omega)} = {\nu(\Omega)}(\lambda_{1,f}-\lambda_{1,f,\gamma})+\lambda_{1,f,\gamma}
\end{multline*}
where ${\nu(\Omega)}=1-\int_{\Omega}f u_{0}^{2}dx<1$. Last inequality can hold only if ${\nu(\Omega)}=0$, that is $\int_{\Omega}f u_{0}^{2}dx=1$. This means that $u_{0}$ is a minimizer for $\lambda_{1,f,\gamma}$, having weighted $L^{2}$ norm equal to $1$ and
\[
\lambda_{1,f,\gamma}=\lim \left[ \int_{\Omega}|\nabla u_{n}|^{2} dx +\gamma\int_{\de\Omega} u_{n}^{2}d\sigma \right]\ge \int_{\Omega} |\nabla u_{0}|^{2}dx+\gamma \int_{\de\Omega}u_{0}^{2}d\sigma \ge \lambda_{1,f,\gamma},
\]
That is $u_{0}$ is a minimizer for $\lambda_{1,f,\gamma}$.

If in addition $f \in L^{\nicefrac N2}(\Omega)$, as {shown} in Remark \ref{2.3}, then $\Sigma_{f}=\emptyset$. Hence (without assuming the smallness of $\lambda_{1,f,\gamma}$) it is not difficult to show that the convergence of the approximating sequence is strong for {$fu_{n}^{2}$} and this {allows us} to pass to the limit and get that $u_{0}\not\equiv 0$ is a minimizer.

\textit{Case $\gamma<0$.} For a normalized minimizing sequence $u_{n}$, {up to a subsequence $u_n\to u_0$ a.e. in $\Omega$, and Fatou's Lemma gives $\int_\Omega fu_0^2dx\le1$. By the weak convergence in $H^1(\Omega)$, the strong convergence of the traces and the definition of $\lambda_{1,f,\gamma}$,}
\begin{align*}
\lambda_{1,f,\gamma}(\Omega) &=\lim_{n\to \infty}\int_{\Omega}|\nabla(u_{n}-u_{0})|^{2}dx +\int_{\Omega}|\nabla u_{0}|^2 d x+\gamma\int_{\partial \Omega}  u_{0}^2 d \sigma \\ &\ge \lim_{n\to \infty}\int_{\Omega}|\nabla(u_{n}-u_{0})|^{2}dx+\lambda_{1,f,\gamma}(\Omega)\int_{\Omega}u_{0}^{2}fdx
\end{align*}
Hence
\[
0\ge \lambda_{1,f,\gamma}(\Omega)\left(1-\int_{\Omega}u_{0}^{2}f dx \right)\ge  \lim_{n\to \infty}\int_{\Omega}|\nabla(u_{n}-u_{0})|^{2}dx,
\]
and all the above inequalities are equalities. This implies that {$\int_\Omega fu_0^2dx=1$,} $u_{n}\to u_{0}$ strongly in $H^{1}$, and $u_{0}{\not\equiv0}$ is a minimum for $\lambda_{1,f,\gamma}(\Omega)$.
\medskip

To conclude the proof of Step 3, we show that for any $\gamma\in \R$, any minimizer $u$ is not identically zero on $\de\Omega$. Since $|u|$ is a minimizer too, we may assume $u\ge0$. The function $u$ is a solution of the Euler-Lagrange equation and satisfies \eqref{eq:EL1}; choosing $\phi\equiv1$ as test function, we get
\[
\gamma\int_{\partial\Omega}u\,d\sigma=\lambda_{1,f,\gamma}\int_\Omega fu\,dx .
\]
If $\gamma\neq0$, then $\lambda_{1,f,\gamma}\neq0$ has the sign of $\gamma$, and $\int_\Omega fu\,dx>0$; hence $\int_{\partial\Omega}u\,d\sigma>0$, and the trace of $u$ does not vanish identically. If $\gamma=0$, the minimizers are the nonzero constants.

As a consequence, if $U$ is a normalized minimizer for $\lambda_{1,f,\gamma}$, then, for every $\gamma'<\gamma$,
\[
\lambda_{1,f,\gamma'}\le R_{\gamma',f}[U]=\lambda_{1,f,\gamma}-(\gamma-\gamma')\int_{\partial\Omega}U^2d\sigma<\lambda_{1,f,\gamma}\le\lambda_{1,f},
\]
that is, \eqref{condg} holds for every $\gamma'<\gamma$. In particular, if $f\in L^{\nicefrac N2}(\Omega)$, minimizers exist for every $\gamma$, and then \eqref{condg} holds for every $\gamma\in\R$.

\textit{\underline{Step 4:} proof of (i)(2)}. The limit \eqref{liminfty2} is given by the inequality $\lambda_{1,f,\gamma}(\Omega)\le \frac{ P(\Omega)}{\|f\|_{L^{1}}}\gamma$. Moreover, it is obvious, by definition, that $\lambda_{1,f,\gamma}$ is nondecreasing in $\gamma$. As regards \eqref{liminfty}, by \eqref{ineq1} we have only to show that
\[
L=\lim_{\gamma\to \infty} \lambda_{1,{f,\gamma}}(\Om) \ge \lambda_{1,f}(\Om).
\]
We proceed by contradiction. Assume that $L < \lambda_{1,f}(\Omega)$. This assumption implies that for $\gamma$ large enough, $\lambda_{1,f,\gamma} < \lambda_{1,f}$, which guarantees that a minimizer for $\lambda_{1,f,\gamma}$ exists.

For each $n \in \mathbb{N}$, we can find a $\gamma_n$ (with $\gamma_n \to \infty$) such that $\lambda_{1,f,\gamma_n} < L + 1/n$. By the definition of the infimum, for each $\gamma_n$, we can find a function $w_n \in H^1(\Omega)$ such that $\displaystyle\int_\Omega f(x) w_n^2 dx = 1$ and
   \[
   \displaystyle\frac{\displaystyle\int_\Omega |\nabla w_n|^2dx+\gamma_{n} \int_{\partial \Omega}w_n^{2}d\sigma}{\displaystyle\int_{\Omega}f(x)w_n^2dx} < L + \frac{1}{n}.
   \]
 Thus, the sequence $\{w_n\}$ is bounded in $H^1(\Omega)$, and, up to a subsequence, $w_n \rightharpoonup w_0$ weakly in $H^1(\Omega)$, $w_n \to w_0$ strongly in $L^2(\Omega)$ and in $L^2(\partial\Omega)$, respectively.
Furthermore, being $\gamma_{n}\to+\infty$, the weak limit $w_0$ belongs to $H^1_0(\Omega)$ {(indeed, $\gamma_n\int_{\partial\Omega}w_n^2d\sigma\le L+1$, so that $\int_{\partial\Omega}w_0^2d\sigma=\lim_n\int_{\partial\Omega}w_n^2d\sigma=0$)}.

Now, let $v_n = w_n - w_0$; then $v_n \rightharpoonup 0$. {By the weak convergence in $H^1(\Omega)$, as in Step 3,}
\begin{align*}
\int_\Omega |\nabla w_n|^2 dx &= \int_\Omega |\nabla w_0|^2 dx + \int_\Omega |\nabla v_n|^2 dx + o(1), \\
1 = \int_\Omega f w_n^2 dx &= \int_\Omega f w_0^2 dx + \int_\Omega f v_n^2 dx + o(1).
\end{align*}
The bound on the gradient of $w_{n}$ gives that
\[
\int_\Omega |\nabla w_0|^2 dx + \int_\Omega |\nabla v_n|^2 dx \le L + o(1).
\]
Since $w_0 \in H^1_0(\Omega)$:
\[
\displaystyle\int_\Omega |\nabla w_0|^2 dx \ge \lambda_{1,f} \int_\Omega f w_0^2 dx;
\]
 moreover, using {\eqref{eqqq} and Lemma \ref{cclemma} (3), which give $\liminf_n\int_\Omega|\nabla v_n|^2dx\ge\mu(\Omega)\ge\lambda_{1,f}\,\nu(\Omega)=\lambda_{1,f}\lim_n\int_\Omega fv_n^2dx$, we have}
 \[
 \displaystyle\int_\Omega |\nabla v_n|^2 dx \ge (\lambda_{1,f} - o(1)) \int_\Omega f v_n^2 dx.
 \]
Substituting,
\[
\lambda_{1,f} \left( \int_\Omega f w_0^2 dx + \int_\Omega f v_n^2 dx \right) \le L + o(1)
\]
and by the decomposition of the denominator
\[
\lambda_{1,f}(1 - o(1)) \le L + o(1).
\]
Taking the limit as $n \to \infty$, we obtain $\lambda_{1,f} \le L$.
This contradicts the initial assumption that $L < \lambda_{1,f}$, and the proof of this step is concluded.

\medskip

\textit{\underline{Step 5:} proof of (iii)(a)}. Let $\lambda(\gamma) := \lambda_{1,f,\gamma}$. {For every fixed $\psi\in H^1(\Omega)\setminus\{0\}$, the map $\gamma\mapsto R_{\gamma,f}[\psi]$ is affine and non-decreasing. Hence $\lambda$, being the infimum of a family of affine non-decreasing functions, is concave and non-decreasing on $\R$; since it is finite by Step 1, it is continuous on $\R$. This proves (iii)(a).

In view of Step 6, we also show that the map $\gamma \mapsto u_\gamma$ is continuous with values in $H^1(\Omega)$ in the open set where \eqref{condg} holds. Let $\gamma_0$ satisfy \eqref{condg}, and let} $\{\gamma_n\}_{n\in\mathbb{N}}$ be any sequence such that $\gamma_n \to \gamma_0$. {By continuity, \eqref{condg} holds at $\gamma_n$ for $n$ large.} Let $w_n$ be the corresponding normalized {nonnegative} minimizer, which exists by {(ii)}. {By the continuity of $\lambda$, $\lambda(\gamma_n)\to L:=\lambda(\gamma_0)$.}

The sequence $\{w_n\}$ is bounded in $H^1(\Omega)$. Let $w_0 \in H^1(\Omega)$ its weak limit. The proof runs similarly as before.

First, let us show that $w_0\not\equiv 0$. We argue by contradiction, assuming $w_0 \equiv 0$. In this case, the mass decomposition $1 = \int_\Omega f w_n^2 dx = \int_\Omega f w_0^2 dx + \nu(\Omega) + o(1)$ {(see Lemma \ref{cclemma} (3))} implies that  $\nu(\Omega)=1$.
The energy limit is
\[
L = \lim_{n\to\infty} \lambda(\gamma_n) = \lim_{n\to\infty} \left( \int_\Omega |\nabla w_n|^2 dx + \gamma_n \int_{\partial\Omega} w_n^2 d\sigma \right).
\]
Since $w_n \rightharpoonup 0$, the trace converges strongly in $L^{2}$ to 0, and the boundary term vanishes in the limit. The energy decomposition becomes $L {\ge} \mu(\Omega)$.
Hence, using {\eqref{eqqq}} we get
\[
L {\ge} \mu(\Omega) \ge \lambda_{1,f}(\Omega) \nu(\Omega) = \lambda_{1,f}(\Omega).
\]
{Hence $\lambda(\gamma_0)=L\ge\lambda_{1,f}(\Omega)$, which contradicts \eqref{condg}.} Thus, the assumption $w_0 \equiv 0$ is false.

Second, we prove that the convergence is strong. Indeed, we first observe that
\[
L = \lim_{n\to\infty} \lambda(\gamma_n){\ge} \left( \int_\Omega |\nabla w_0|^2 dx + \gamma_0 \int_{\partial\Omega} w_0^2 d\sigma \right) + \mu(\Omega).
\]
By definition  of $\lambda(\gamma_{0})$ and the concentration-{compactness} Lemma:
\[
L \ge \left( \lambda(\gamma_0) \int_\Omega f w_0^2 dx \right) + \lambda_{1,f} \nu(\Omega).
\]
Using the mass decomposition $\int_\Omega f w_0^2 dx = 1 - \nu(\Omega)$, and the fact that $L {=} \lambda(\gamma_0)$:
\begin{align*}
\lambda(\gamma_0) &\ge \lambda(\gamma_0) (1 - \nu(\Omega)) + \lambda_{1,f} \nu(\Omega) \\
0 &\ge (\lambda_{1,f} - \lambda(\gamma_0))\nu(\Omega).
\end{align*}
Since $\lambda(\gamma_0) < \lambda_{1,f}$ by hypothesis, the term $(\lambda_{1,f} - \lambda(\gamma_0))$ is strictly positive. Since $\nu(\Omega) \ge 0$, this inequality can only hold if $\nu(\Omega)=0${, that is $\int_\Omega fw_0^2dx=1$. Then $\int_\Omega|\nabla w_0|^2dx+\gamma_0\int_{\partial\Omega}w_0^2d\sigma\ge\lambda(\gamma_0)=L$, and, since (up to a subsequence)
\[
L=\lim_{n\to\infty}\lambda(\gamma_n)=\int_\Omega|\nabla w_0|^2dx+\gamma_0\int_{\partial\Omega}w_0^2d\sigma+\lim_{n\to\infty}\int_\Omega|\nabla(w_n-w_0)|^2dx,
\]
we get $\int_\Omega|\nabla(w_n-w_0)|^2dx\to0$}. This implies that the convergence is strong: $w_n \to w_0$ in $H^1(\Omega)$.

Finally, we can pass to the limit in all terms, showing that
 $\displaystyle\int_\Omega f w_0^2 dx = \lim_{n\to\infty} \int_\Omega f w_n^2 dx = 1$, and
    \[
    \lim_{n\to\infty} \lambda (\gamma_{n}) = \lim_{n\to\infty} \left( \int_\Omega |\nabla w_n|^2 dx + \gamma_n \int_{\partial\Omega} w_n^2 d\sigma \right) = \int_\Omega |\nabla w_0|^2 dx + \gamma_0 \int_{\partial\Omega} w_0^2 d\sigma.
    \]
Hence
\[
 \frac{\displaystyle\int_\Omega |\nabla w_0|^2 dx + \gamma_0 \int_{\partial\Omega} w_0^2 d\sigma}{\displaystyle\int_\Omega f w_0^2 dx} = L.
\]
{that is, $w_0$ is a nonnegative normalized minimizer for $\lambda(\gamma_0)$. By Proposition 3.1, $w_0=u_{\gamma_0}$; since the limit does not depend on the subsequence, $u_{\gamma_n}\to u_{\gamma_0}$ strongly in $H^1(\Omega)$.}

\medskip

\textit{\underline{Step 6}:} \textit{proof of (iii)(b)}. Let us still denote $\lambda(\gamma) := \lambda_{1,f,\gamma}${, and let $\gamma$ satisfy \eqref{condg}. By Step 5, \eqref{condg} holds in a neighborhood of $\gamma$, where the map $\gamma \mapsto u_\gamma$ is continuous with values in $H^1(\Omega)$.}

Let $u_\gamma\ge0$ normalized first eigenfunction for $\lambda(\gamma)$. Using $u_{\gamma}$ as a test function for $\lambda(\gamma+h)$, we get
\[
\lambda(\gamma+h) \le \frac{\displaystyle\int_\Omega |\nabla u_\gamma|^2 dx + (\gamma+h)\int_{\partial\Omega} u_\gamma^2 d\sigma}{\displaystyle\int_\Omega f u_\gamma^2 dx}=\lambda(\gamma) + h\int_{\partial\Omega} u_\gamma^2 d\sigma,
\]
and then
\[
\limsup_{h\to 0^{+}}\frac{\lambda(\gamma+h)-\lambda(\gamma)}{h} \le \int_{\de\Omega}u_{\gamma}^{2}d\sigma.
\]
Similarly, we have
\begin{equation*}
\lambda(\gamma) \le \lambda(\gamma+h) - h\int_{\partial\Omega} u_{\gamma+h}^2 d\sigma;
\end{equation*}
rearranging and dividing by $h > 0$ we have
\[
\int_{\partial\Omega} u_{\gamma+h}^2 d\sigma \le \frac{\lambda(\gamma+h) - \lambda(\gamma)}{h}.
\]
Using the continuity of the map $\gamma \mapsto u_\gamma$, we have that $u_{\gamma+h} \to u_\gamma$ strongly in $H^1(\Omega)$. This implies strong convergence of the traces in $L^2(\partial\Omega)$, and therefore:
\[
\int_{\partial\Omega} u_\gamma^2 d\sigma = \lim_{h\to 0^+} \int_{\partial\Omega} u_{\gamma+h}^2 d\sigma \le \liminf_{h\to 0^+} \frac{\lambda(\gamma+h) - \lambda(\gamma)}{h}.
\]
 An identical argument holds for the left-hand derivative. {This proves (iii)(b).

Finally, if $f\in L^{\nicefrac{N}{2}}(\Omega)$, by Step 3 condition \eqref{condg} holds for every $\gamma\in\R$, so that, by (iii)(b), $\lambda'(\gamma)=\int_{\partial\Omega}u_\gamma^2d\sigma>0$ for every $\gamma$, the trace of the minimizers being nonzero (see Step 3). Hence $\lambda$ is strictly increasing on $\R$, which proves (iii)(c).} This concludes the proof.
\end{proof}

\begin{remark}
\label{rem3.3}
We emphasize that the statement \textit{(iii), (b)} of the previous proposition{, together with the fact that the minimizers have nonzero trace (see \textit{(ii)}),} implies the existence of a threshold { 
\[
\bar\gamma=\sup\left\{\gamma\in\R\colon\lambda_{1,f,\gamma}<\lambda_{1,f}\right\}\in (0, +\infty],
\]
} such that the map is {strictly} increasing for $\gamma \le  \bar{\gamma}$ and constant for $\gamma \ge \bar{\gamma}$ (where $\lambda_{1,f,\gamma} = \lambda_{1,f}$). Indeed, by monotonicity and continuity, \eqref{condg} holds if and only if $\gamma<\bar\gamma$, and $\lambda'(\gamma)>0$ in $(-\infty,\bar\gamma)$ by {\it(iii)(b)}.

As regards the existence of minimizers, by {\it(ii)} the infimum is attained for every $\gamma<\bar\gamma$. On the other hand, it is not attained for any $\gamma>\bar\gamma$: if $u$ were a normalized minimizer for some $\gamma'>\bar\gamma$, then, for $\gamma\in(\bar\gamma,\gamma')$, since $u$ has nonzero trace,
\[
\lambda_{1,f,\gamma}\le R_{\gamma,f}[u]=\lambda_{1,f,\gamma'}-(\gamma'-\gamma)\int_{\partial\Omega}u^2d\sigma<\lambda_{1,f},
\]
against the definition of $\bar\gamma$. Hence \eqref{condg} is equivalent to the existence of a minimizer, except possibly for $\gamma=\bar\gamma<+\infty$. As shown in {Examples \ref{ex:hardy} and \ref{rem:corner}}, if $f\not\in L^{\nicefrac{N}{2}}$, $\bar\gamma$ may be finite. {In this case, the infimum may or may not be attained at $\gamma=\bar\gamma$: it is not attained in Example \ref{ex:hardy}, while in Example \ref{rem:corner} it is attained, although \eqref{condg} fails, and $\lambda$ is not differentiable at $\bar\gamma$.}
\end{remark}

\begin{example}
\label{ex:hardy}
Let $\Omega = B_1(0) \subset \mathbb{R}^N$ with $N \ge 3$, and let the weight function be the Hardy potential, $f(x) = |x|^{-2}$. For any $\gamma \in \mathbb{R}$, the first eigenvalue of the problem
\[
\lambda_{1,f,\gamma} = \inf_{\psi \in H^1(B_1)\setminus\{0\}} \dfrac{\ds\int_{B_1} |\nabla \psi|^2dx+\gamma \int_{\de B_1}\psi^{2}d\sigma}{\ds\int_{B_1}\frac{\psi^2}{|x|^2}dx}
\]
is given as follows:
\begin{enumerate}[(i)]
    \item If $\gamma < \bar\gamma:=\frac{N-2}{2}$, the infimum is attained, and the eigenvalue is
    \[
        \lambda_{1,f,\gamma} = \gamma(N-2-\gamma).
    \]
    The corresponding eigenfunction (unique up to a multiplicative constant) is
    \[
        u_\gamma(x) = |x|^{-\gamma}.
    \]
    \item If $\gamma \ge \bar\gamma$, the eigenvalue is equal to the Hardy constant,
    \[
        \lambda_{1,f,\gamma} = \left(\frac{N-2}{2}\right)^2,
    \]
    and the infimum is not attained by any function in $H^1(B_1)$.
\end{enumerate}
\end{example}

\begin{proof}
{Let $\gamma<\frac{N-2}{2}$. Then $|x|^{-\gamma}\in H^1(B_1)$, and a direct computation gives 
\begin{equation}
\label{lambdapass}
\lambda\le R_{\gamma,f}[|x|^{-\gamma}]=\gamma(N-2-\gamma)<\left(\frac{N-2}{2}\right)^2=\lambda_{1,f}(B_1),
\end{equation} 
with $\lambda=\lambda_{1,f,\gamma}$; hence \eqref{condg} holds and, by Proposition \ref{prop:eigenvalue_properties} {\it(ii)}, the infimum is attained (notice that $\Sigma_f=\{0\}$). By Proposition \ref{simplicity} the minimizer is unique up to a multiplicative constant; hence, by} {symmetry} of the domain, the eigenfunction $u_\gamma$ must be radial, $u_\gamma(x) = u(r)$ where $r=|x|$. The corresponding Euler-Lagrange equation is
\[
u''(r) + \frac{N-1}{r} u'(r) + \frac{\lambda}{r^2} u(r) = 0,
\]
subject to the Robin boundary condition $u'(1) + \gamma u(1) = 0$.
We look for solutions of the form $u(r) = r^\alpha$. Substituting this into the ODE, we obtain
\[
\alpha^2 + (N-2)\alpha + \lambda = 0.
\]
The two roots are $\alpha_{1,2} = \frac{-(N-2) \pm \sqrt{(N-2)^2 - 4\lambda}}{2}$ {(real and distinct by \eqref{lambdapass})}, and the general solution is $u(r) = c_1 r^{\alpha_1} + c_2 r^{\alpha_2}$.

In order to get $u\in H^{1}$, we require the exponents to be greater than $\frac{2-N}{2}$. Since $\alpha_2 < \frac{2-N}{2}$, we must set $c_2=0$, and $u(r) = A r^{\alpha_1}$. Using the Robin condition $u'(1) + \gamma u(1) = 0$, we obtain $\alpha_1 = -\gamma$, and then $ \gamma < \frac{N-2}{2}$.

This analysis confirms that for any $\gamma < \frac{N-2}{2}$ (including all $\gamma \le 0$), a minimizer exists.
We can now find the eigenvalue $\lambda$ by substituting $\alpha = \alpha_1 = -\gamma$, obtaining
\[
 \lambda = \gamma(N-2-\gamma).
\]
This proves {\it(i)}.

If $\gamma \ge (N-2)/2$, no solution of the form $r^\alpha$ can satisfy both the $H^1(B_1)$ requirement and the boundary condition. Since the map $\gamma \mapsto \lambda_{1,f,\gamma}$ is non-decreasing and continuous, and $\lambda_{1,f,\gamma} \le \lambda_{1,f} = (\frac{N-2}{2})^2$, we must have:
\[
\lambda_{1,f,\gamma} = \left(\frac{N-2}{2}\right)^2, \text{ for all }\gamma \ge \frac{N-2}{2}.
\]
In this case, the infimum is not attained. Indeed, by contradiction, let assume there exists $u_{\gamma_1} \in H^1(B_1)$ such that $R_{\gamma_1}[u_{\gamma_1}] = (\frac{N-2}{2})^2$. Since $\lambda(\gamma)$ is strictly increasing for $\gamma < (N-2)/2$, $\gamma_1$ cannot be an interior point of the interval where $\lambda(\gamma)$ is constant. If $\gamma_1 > (N-2)/2$, we could choose $\gamma_2$ such that $(N-2)/2 \le \gamma_2 < \gamma_1$. Then
\[
R_{\gamma_2}[u_{\gamma_1}] = R_{\gamma_1}[u_{\gamma_1}] - (\gamma_1-\gamma_2)\int_{\partial\Omega} u_{\gamma_1}^2 d\sigma = \left(\frac{N-2}{2}\right)^2 - (\gamma_1-\gamma_2)\int_{\partial\Omega} u_{\gamma_1}^2 d\sigma.
\]
If the trace is not zero, then $R_{\gamma_2}[u_{\gamma_1}] < (\frac{N-2}{2})^2 = \lambda(\gamma_2)$, which is impossible. If the trace is zero, $u_{\gamma_1} \in H^1_0(B_1)$, which is known to not contain a minimizer for the Hardy quotient. {Finally, if $\gamma_1=\frac{N-2}{2}$, arguing as in the proof of {\it(i)}, $u_{\gamma_1}$ would be a positive radial solution in $H^1(B_1)$ of the ODE with $\lambda=(\frac{N-2}{2})^2$; in this case $\alpha_1=\alpha_2=-\frac{N-2}2$, and the solutions are $r^{-\frac{N-2}{2}}(c_1+c_2\log r)$, which do not belong to $H^1(B_1)$ unless $c_1=c_2=0$.} This proves {\it(ii)}.
\end{proof}

\begin{example}[The infimum can be attained at $\bar\gamma$]\label{rem:corner}

Let $N\ge5$, $a=\frac{N-2}{2}$, $\Omega=B_1$ and
\[
f(x)=\frac{1}{|x|^2}\Big(1-\frac{2}{a^2L(x)^2}\Big),\qquad u(x)=\frac{|x|^{-a}}{L(x)},\qquad\text{where }L(x)=1-\log|x|.
\]
Since $a^2>2$ and $L\ge1$, we have $0<f\le|x|^{-2}$, so that $f\in M^{\nicefrac{N}{2}}(B_1)$ and $\Sigma_f=\{0\}$, and Proposition \ref{prop:eigenvalue_properties} applies.
\begin{enumerate}[(a)]
\item $\lambda_{1,f}(B_1)=a^2$, and it is not attained. Indeed, for any $\psi\in H^1_0(B_1)\setminus\{0\}$, since $f<|x|^{-2}$ in $B_1$, by the Hardy inequality we have
\[
a^2\int_{B_1}f\psi^2dx<a^2\int_{B_1}\frac{\psi^2}{|x|^2}dx\le\int_{B_1}|\nabla\psi|^2dx,
\]
so that the Rayleigh quotient of $\psi$ is greater than $a^2$. On the other hand, for $\eps>0$ we have $f\ge(1-\eps)|x|^{-2}$ in a small ball $B_\delta$, and the Hardy constant of $B_\delta$ is still $a^2$, so that $\lambda_{1,f}(B_1)\le a^2/(1-\eps)$.
\item $\lambda_{1,f,\bar\gamma}(B_{1})=a^{2}$, and it is attained by $u$, for $\bar\gamma=a-1=\frac{N-4}{2}$. Indeed:
\begin{enumerate}[(b1)]
\item $u\in H^1(B_1)$, since $|\nabla u|^2\le C|x|^{-N}L^{-2}$ and $\int_0^1\frac{dr}{r(1-\log r)^2}<\infty$. Moreover, a direct computation gives
\begin{equation}
\label{robex}
-\Delta u=a^2fu\ \text{ in }B_1\setminus\{0\},\qquad \frac{\de u}{\de\nu}+\bar\gamma\,u=0\ \text{ on }\de B_1.
\end{equation}
\item For $\psi\in C^1(\overline B_1)$, we integrate Picone's identity $|\nabla\psi|^2-\nabla u\cdot\nabla(\psi^2/u)=u^2|\nabla(\psi/u)|^2$ on $B_1\setminus B_\eps$. By the divergence theorem, the equation and the boundary condition satisfied by $u$, we have
\[
\int_{B_1\setminus B_\eps}\nabla u\cdot\nabla\Big(\frac{\psi^2}{u}\Big)dx=a^2\int_{B_1\setminus B_\eps}f\psi^2dx-\bar\gamma\int_{\de B_1}\psi^2d\sigma-\int_{\de B_\eps}\frac{\psi^2}{u}\,\frac{\de u}{\de r}\,d\sigma,
\]
and, since $|\de_ru|/u\le a/\eps$ on $\de B_\eps$, the last term goes to zero as $\eps\to0$. We then get
\[
\int_{B_1}|\nabla\psi|^2dx+\bar\gamma\int_{\de B_1}\psi^2d\sigma-a^2\int_{B_1}f\psi^2dx=\int_{B_1}u^2\Big|\nabla\Big(\frac\psi u\Big)\Big|^2dx\ge0 .
\]
By density, since the left-hand side is continuous on $H^1(B_1)$ (by the Hardy inequality, being $f\le|x|^{-2}$, and the continuity of the trace), it is nonnegative for every $\psi\in H^1(B_1)$. Moreover, it vanishes for $\psi=u$: multiplying \eqref{robex} by $u$ and integrating on $B_1\setminus B_\eps$, the boundary term on $\de B_\eps$ goes to zero as $\eps\to0$.
\end{enumerate}
\end{enumerate}
Hence $\lambda_{1,f,\gamma}=a^2=\lambda_{1,f}$ for every $\gamma\ge\bar\gamma$, and the infimum is attained for $\gamma=\bar\gamma$, although \eqref{condg} fails. For $\gamma<\bar\gamma$, instead, $\lambda_{1,f,\gamma}<a^{2}$ by Proposition \ref{prop:eigenvalue_properties}(ii).

Therefore $\bar\gamma=\frac{N-4}{2}$ is the threshold of Remark \ref{rem3.3}. The map $\lambda$ is constant on $[\bar\gamma,+\infty)$, while its left derivative at $\bar\gamma$ is greater than or equal to $\int_{\de B_1}u^2d\sigma\big/\int_{B_1}fu^2dx>0$; hence $\lambda$ is not differentiable at $\bar\gamma$, and the assumption \eqref{condg} in Proposition \ref{prop:eigenvalue_properties} (iii)(b) cannot be dropped.

\end{example}

\begin{remark}[The case $\gamma\le0$]\label{rem:gammaneg}

Steps 1 and 2 and the case $\gamma<0$ of Step 3 use neither $(H_{f,1})$--$(H_{f,2})$ nor Lemma \ref{cclemma}. Hence, if $\gamma\le0$, $\lambda_{1,f,\gamma}(\Omega)$ is attained for every $f\in M^{\nicefrac N2}(\Omega)$ with $f>0$ a.e. in $\Omega$ (for $\gamma=0$ the minimizers are the nonzero constants), and $(H_{f,1})$--$(H_{f,2})$ are needed only for $\gamma>0$, where \eqref{condg} is not automatically satisfied.

\end{remark}

\begin{remark}[Singularities on $\partial\Omega$]\label{rem:boundary}

Hypothesis $(H_{f,2})$ cannot be dropped in \eqref{liminfty}. Let $\Omega=B_1^+=B_1\cap\{x_N>0\}$ and $f(x)=|x|^{-2}$. Then $f\in M^{\nicefrac{N}{2}}(\Omega)$, $f>0$, and $\Sigma_f=\{0\}\subset\partial\Omega$, so that $(H_{f,1})$ holds but $(H_{f,2})$ does not.

On the one hand, $\lambda_{1,f}(\Omega)=\frac{N^2}{4}$, since it is the sharp constant in the Hardy inequality in the half-space $\R^N_+=\{x_N>0\}$ with singularity at a boundary point,
\[
\int_{\R^N_+}|\nabla\psi|^2dx\ge\frac{N^2}{4}\int_{\R^N_+}\frac{\psi^2}{|x|^2}dx\qquad\text{for every }\psi\in C^\infty_0(\R^N_+),
\]
and the Rayleigh quotient is invariant under dilations.

On the other hand, let $\psi\in C^\infty_0(B_1)$ be radial, and $\psi_n(x)=\psi(nx)$. Since $\operatorname{supp}\psi_n\subset B_{1/n}$, only the flat part of $\partial\Omega$ is involved, and a change of variables gives, by the symmetry of $\psi$,
\[
R_{\gamma,f}[\psi_n]=\frac{\displaystyle\int_{B_1^+}|\nabla\psi|^2dx+\frac{\gamma}{n}\int_{B_1\cap\{x_N=0\}}\psi^2dx'}{\displaystyle\int_{B_1^+}\frac{\psi^2}{|x|^2}dx}\ \to\ \frac{\displaystyle\int_{B_1}|\nabla\psi|^2dx}{\displaystyle\int_{B_1}\frac{\psi^2}{|x|^2}dx}\quad\text{as }n\to\infty.
\]
Taking the infimum over radial $\psi$, we obtain $\lambda_{1,f,\gamma}(\Omega)\le(\frac{N-2}{2})^2$ for every $\gamma\in\R$. Therefore
\[
\lim_{\gamma\to+\infty}\lambda_{1,f,\gamma}(\Omega)\le\left(\frac{N-2}{2}\right)^2<\frac{N^2}{4}=\lambda_{1,f}(\Omega),
\]
and \eqref{liminfty} fails.

\end{remark}

\begin{remark}[Sign of $f$]\label{rem:fpos}

For $\gamma\ge0$, all the results of this section hold weakening the sign condition on $f$, precisely for $f\ge0$, $f\not\equiv0$. Indeed, the assumption $f>0$ a.e. is used only in two points.
\begin{enumerate}[(a)]
\item In Proposition 2.2. This is used only when $\gamma<0$, in Step 1, to show that $\lambda_{1,f,\gamma}>-\infty$, and to bound the minimizing sequences in $H^1(\Omega)$. When $\gamma\ge0$ neither is needed. Indeed, $\lambda_{1,f,\gamma}\ge0$ trivially, and the bound in $H^1(\Omega)$ follows from Proposition \ref{propequiv}, which holds for $f\ge0$, $f\not\equiv0$. For $\gamma<0$ the assumption cannot be dropped. For instance, if $f=0$ in $\Omega\cap B$, where $B$ is a ball centered at a point of $\partial\Omega$, choose $\varphi$ supported in $B$ with $\int_\Omega|\nabla\varphi|^2dx+\gamma\int_{\partial\Omega}\varphi^2d\sigma<0$, which is possible for $\gamma$ sufficiently negative, and $\eta\in C^\infty_0(\Omega\setminus\overline B)$ with $\int_\Omega f\eta^2dx>0$. Then $R_{\gamma,f}[\varphi+t\eta]\to-\infty$ as $t\to0$, that is, $\lambda_{1,f,\gamma}=-\infty$.
\item In the proof of Proposition 3.1, in order to have $\int_\Omega f(w^+)^2dx>0$ when $w^+\not\equiv0$. For $\gamma\ge0$ this is not needed. If $w^+\not\equiv0$ and $\int_\Omega f(w^+)^2dx=0$, then 
\[
\int_\Omega|\nabla w^+|^2dx+\gamma\int_{\partial\Omega}(w^+)^2d\sigma=0,
\]
 so that $w^+$ is a nonnegative constant $c$. Moreover, $c=0$: if $\gamma>0$ because the trace of $w^+$ vanishes, and if $\gamma=0$ because $c^2\int_\Omega f\,dx=0$. This is a contradiction. In the last step of the proof, $f\not\equiv0$ is enough.
\end{enumerate}

\end{remark}

\section{Existence results for Robin problems with natural growth in the gradient}
We start observing that a smallness assumption on $\lambda f$ is needed in order to get existence.
\begin{prop}\label{prop:necessary}
Let $\Omega \subset \mathbb{R}^N$ be a bounded {Lipschitz} domain.
Let $f \in M^{\nicefrac{N}{2}}(\Omega)$ be a weight function such that {$f\ge0$, $f\not\equiv 0$} in $\Omega$. Suppose there exists a  solution $u \in H^1(\Omega)$ to
\begin{equation}
\label{problemmainnec}
\begin{cases}
    \Delta u + {\sigma_0}|\nabla u|^2 + \lambda f = 0 & \text{in } \Omega \\
    \frac{\partial u}{\partial \nu} + \beta u = 0 & \text{on } \partial\Omega
\end{cases}
\end{equation}
for some constants $\lambda > 0$, {$\sigma_0>0$ and $\beta\in\R$}.
Then
\[
    \lambda \le {\frac{\lambda_{1,f}(\Omega)}{\sigma_0}}.
\]
\end{prop}

\begin{proof}
The proof is by contradiction.
Assume that a solution $u \in H^1(\Omega)$ of \eqref{problemmainnec} exists for some $\lambda > {\lambda_{1,f}(\Omega)/\sigma_0}$. Then, for every test function $\psi \in H^1(\Omega)\cap L^{\infty}(\Omega)$:
\begin{equation}
\label{wfu}
\int_\Omega \nabla u \cdot \nabla\psi \,dx + \beta \int_{\partial\Omega}  u \psi \,d\sigma = \int_\Omega ({\sigma_0}|\nabla u|^2 + \lambda f)\psi \,dx.
\end{equation}
The assumption ${\sigma_0}\lambda > \lambda_{1,f}(\Omega)$ implies that there exists a function $\varphi_0 \in C_{0}^\infty(\Omega)$ {such that}
\begin{equation}
 \label{eq:key_inequality2}
\int_{\Omega} |\nabla \varphi_{0}|^{2} dx < {\sigma_0}\lambda \int_{\Omega} f \varphi_{0}^{2}dx\,.
\end{equation}
We now choose $\psi = \varphi_0^2$ as a test function in \eqref{wfu}.
It follows that, {by the elementary inequality $2ab-\sigma_0b^2\le a^2/\sigma_0$},
\begin{equation*}
\lambda \int_\Omega f \varphi_0^2 \,dx = \int_\Omega \left( 2\varphi_0 (\nabla u \cdot \nabla\varphi_0) - {\sigma_0}|\nabla u|^2 \varphi_0^2 \right) \,dx
\le {\frac{1}{\sigma_0}}\int_{\Omega} |\nabla \varphi_{0}|^{2}dx.
\end{equation*}
This contradicts \eqref{eq:key_inequality2}.
\end{proof}
Now we finally consider the following class of Robin boundary value problems
\begin{equation}
\label{generalrobinpb}
\begin{cases}-\operatorname{div}(a(x, u,\nabla u))=H(x, u, \nabla u) & \text { in } \Omega \\ a(x,u,\nabla u)\cdot \nu+\beta u=0 & \text { on } \partial \Omega\end{cases}
\end{equation}
where $\Omega$ is a bounded Lipschitz, connected open subset of $\mathbb{R}^N, N \ge 3$, $\nu$ is the outer normal to $\Omega$ on $\de\Omega$ and $\beta > 0$. Moreover, the Carath\'eodory functions
$
a: \Omega \times \R\times \mathbb{R}^N \rightarrow \mathbb{R}^N
$
and
$
H: \Omega \times \R\times \mathbb{R}^N \rightarrow \mathbb{R}
$
satisfy
\begin{equation}\label{hp1a}
{a}(x,u, \xi) \cdot \xi \ge|\xi|^2,\qquad \forall \xi \in \R^{N},
\end{equation}
the monotonicity condition
\begin{equation}\label{hp2a}
\left({a}(x,u, z)-{a}\left(x, u, z^{\prime}\right)\right) \cdot\left(z-z^{\prime}\right)>0, \quad z \ne z^{\prime}
\end{equation}
and the growth conditions
\begin{equation}\label{hp3a}
|a(x,u, z)| \leq a_0|z|+a_1, \quad a_0, a_1>0,
\end{equation}
and
\begin{equation}\label{hpH}
|H(x,u, z)| \leq \sigma_0 |z|^2 + \lambda f(x), \quad \sigma_0 > 0,
\end{equation}
for almost every $x \in {\Omega}$, for every $u\in \R$, and $z, z^{\prime} \in \mathbb{R}^N$. Here $\lambda \ge 0$ is a constant and $f$ is a non-negative function in $M^{\nicefrac{N}{2}}(\Omega)$, $f \not\equiv 0$.

We say that $u\in H^{1}$ is a weak solution to \eqref{generalrobinpb}
if for any $\varphi\in H^{1}(\Omega)\cap L^{\infty}(\Omega)$ it holds
\begin{equation*}\label{defw}
\int_{\Omega} a(x,u,\nabla u)\cdot \nabla \varphi \,dx +\beta\int_{\de \Omega} u\varphi\,d\sigma =
\int_{\Omega} H(x,u,\nabla u)\varphi \,dx.
\end{equation*}

\begin{theorem}
\label{thm:main_existence}
Suppose that $f\in M^{\nicefrac N2}(\Omega)$, {$f\ge 0$ ($f\not\equiv0$)} and $(H_{f,1})-(H_{f,2})$ hold. Let us assume
\[
\lambda < \frac{\lambda_{1,f}(\Omega)}{\sigma_{0}},
\]
where $\lambda_{1,f}(\Omega)$ is defined in \eqref{lambda1f}, and $\sigma_{0}>0$ in \eqref{hpH}. Then there exists at least a solution $u\in H^{1}(\Omega)$ of problem \eqref{generalrobinpb}, under the conditions \eqref{hp1a}, \eqref{hp2a}, \eqref{hp3a}, \eqref{hpH} and $u$ is such that $e^{\sigma_{0}|u|}\in H^{1}(\Omega)$.
\end{theorem}
\begin{proof}

\textit{Step 1: a priori estimates}.
Let us consider the following approximation problem:
\begin{equation}
\label{rbvpapp}
\begin{cases}
    -\operatorname{div}(a(x, u_{n},\nabla u_{n}))=H_{n}(x, u_{n}, \nabla u_{n}) & \text { in } \Omega, \\
    a(x,u_{n},\nabla u_{n})\cdot \nu+\beta u_{n}=0 & \text { on } \partial \Omega,
\end{cases}
\end{equation}
where $H_n$ is a bounded approximation of $H$:
\[
H_{n}=\frac{H}{1+\nicefrac1n|H|},
\]
and it is such that $|H_{n}(x,u,\xi)| \le \min\{n, \sigma_0|\xi|^2 + \lambda f(x)\}$. Classical results can be applied in order to get the existence of a solution $u_{n}\in H^{1}(\Omega)$ of \eqref{rbvpapp} (see \cite{ll}) {(notice that the operator is coercive on $H^1(\Omega)$, since $\beta>0$)}. Moreover, from a classical argument by Stampacchia, $u_{n}\in L^{\infty}(\Omega)$.

Since $u_n$ does not have a constant sign, we split the estimates. We define $u_n^+ = \max\{u_n, 0\}$ and $u_n^- = \max\{-u_n, 0\}$.

\textit{Positive part estimate.} We choose the test function $\varphi_n = e^{2\sigma_0 u_n^+} - 1$. Note that $\varphi_n \ge 0$ and $\varphi_n$ vanishes on $\{u_n \le 0\}$, so the integrals are restricted to $\{u_n > 0\}$. We get:
\begin{multline*}
    \int_\Omega a(x,u_n,\nabla u_n) \cdot \nabla(e^{2\sigma_0 u_n^+} - 1) \,dx + \beta \int_{\partial\Omega} u_n (e^{2\sigma_0 u_n^+} - 1) \,d\sigma = \\ = \int_\Omega H_n(x,u_n,\nabla u_n) (e^{2\sigma_0 u_n^+} - 1) \,dx.
\end{multline*}
Using the ellipticity \eqref{hp1a} and the growth condition on $H_n$ on the set $\{u_n > 0\}$:
\begin{multline*}
    2\sigma_0 \int_{\{u_n > 0\}} |\nabla u_n|^2 e^{2\sigma_0 u_n} \,dx + \beta \int_{\partial\Omega} u_n^+ (e^{2\sigma_0 u_n^+} - 1) \,d\sigma \le \\ \le  \sigma_0 \int_{\{u_n > 0\}} |\nabla u_n|^2 (e^{2\sigma_0 u_n} - 1) \,dx + \lambda \int_\Omega f (e^{2\sigma_0 u_n^+} - 1) \,dx.
\end{multline*}
We perform the change of variables $v_n = e^{\sigma_0 u_n^+}$, obtaining
\begin{equation}
\label{mainineqcv}
    \frac{1}{\sigma_0} \int_\Omega |\nabla v_n|^2 dx + \sigma_0 \int_{\{u_n > 0\}} |\nabla u_n|^2 dx + \frac{\beta}{\sigma_0} \int_{\partial\Omega} (v_n^2-1) \log v_n \, d\sigma \le \lambda \int_\Omega f v_n^2 dx.
\end{equation}
We choose $\gamma > 0$ such that $\sigma_0 \lambda < \lambda_{1,f,\gamma}$ {(this is possible by \eqref{liminfty}, since $\sigma_0\lambda<\lambda_{1,f}(\Omega)$)}. Hence by \eqref{mainineqcv} and the definition of $\lambda_{1,f,\gamma}$ we get
\begin{multline}
\label{mainineqcv2}
     \int_\Omega |\nabla v_n|^2 dx + \beta \int_{\partial\Omega} (v_n^2-1) \log v_n \, d\sigma \le    \\ \le \sigma_0 \lambda \int_\Omega f v_n^2 dx\le \frac{\sigma_0 \lambda}{\lambda_{1,f,\gamma}} \left[ \int_\Omega |\nabla v_n|^2dx+\gamma \int_{\partial \Omega}v_n^2d\sigma\right].
\end{multline}
From \eqref{mainineqcv2} we get
\begin{equation*}
\left(1-\frac{\sigma_0 \lambda}{\lambda_{1,f,\gamma}}\right) \int_{\Omega} |\nabla v_n|^{2} dx\le
\int_{\de\Omega} \left[\left( C_{1} - \beta \log v_{n}\right)v_{n}^{2}+\beta \log v_{n} \right] d\sigma
\end{equation*}
where $C_{1}=\frac{\sigma_0 \lambda \gamma }{\lambda_{1,f,\gamma}}$. Since $\beta > 0$, the integrand on the right-hand side, $g(t) = t^2(C_{1} - \beta\log t)+\beta \log t$, is bounded from above for $t \ge 1$ and the maximum is given by some value $C_{2}=C_{2}(\lambda,\sigma_{0},\gamma,\beta,f)$.
Hence
\begin{equation}
\label{stimaDv22}
\int_{\Omega} |\nabla v_{n}|^{2} dx \le \frac{C_2 P(\Omega)}{1-\nicefrac{\sigma_0 \lambda}{\lambda_{1,f,\gamma}}}.
\end{equation}
Similarly, from \eqref{mainineqcv2}, substituting $\int_\Omega |\nabla v_n|^2 dx \le \sigma_0 \lambda \int_\Omega f v_n^2 dx - \beta \int_{\partial\Omega} (v_n^2-1) \log v_n d\sigma$ into the eigenvalue inequality relation, we can derive a bound for the weighted $L^2$ norm:
\begin{equation}
\label{stimav2}
\int_{\Omega} f v_{n}^{2} dx \le \frac{1}{\lambda_{1,f,\gamma}-\sigma_0 \lambda} \int_{\de\Omega} \left[\left(\gamma -\beta \log v_{n}\right) v_{n}^{2}+\beta\log v_{n}\right]d\sigma \le \frac{C_{3}P(\Omega)}{\lambda_{1,f,\gamma}-\sigma_0 \lambda} ,
\end{equation}
where $C_3$ is the maximum of the function $h(t)=t^2(\gamma-\beta\log t)+\beta \log t$ for $t \ge 1$.

\textit{Negative part estimate.}
We proceed similarly for $u_n^-$. We choose the test function $\psi_n = 1 - e^{2\sigma_0 u_n^-}$. Note that $\psi_n \le 0$ and it vanishes on $\{u_n \ge 0\}$, so the support is $\{u_n < 0\}$. The gradient is $\nabla \psi_n = -2\sigma_0 e^{2\sigma_0 u_n^-} \nabla u_n^- = 2\sigma_0 e^{2\sigma_0 u_n^-} \nabla u_n$ on the set $\{u_n < 0\}$. Substituting into \eqref{rbvpapp}:
\begin{multline*}
    \int_{\{u_n < 0\}} a(x,u_n,\nabla u_n) \cdot (2\sigma_0 e^{2\sigma_0 u_n^-} \nabla u_n) \,dx + \beta \int_{\partial\Omega} u_n (1 - e^{2\sigma_0 u_n^-}) \,d\sigma = \\ = \int_{\{u_n < 0\}} H_n(x,u_n,\nabla u_n) (1 - e^{2\sigma_0 u_n^-}) \,dx.
\end{multline*}
Using the ellipticity \eqref{hp1a} on the left hand side:
\[
2\sigma_0 \int_{\{u_n < 0\}} |\nabla u_n|^2 e^{2\sigma_0 u_n^-} \,dx.
\]
For the boundary term, note that on $\{u_n < 0\}$, $u_n = -u_n^-$, so:
\[
\beta \int_{\partial\Omega} (-u_n^-) (1 - e^{2\sigma_0 u_n^-}) \,d\sigma = \beta \int_{\partial\Omega} u_n^- (e^{2\sigma_0 u_n^-} - 1) \,d\sigma.
\]
For the right hand side, since $1 - e^{2\sigma_0 u_n^-} \le 0$, we use the absolute value of the bound on $H_n$:
\begin{align*}
\int_{\{u_n < 0\}} H_n (1 - e^{2\sigma_0 u_n^-}) \,dx &\le \int_{\{u_n < 0\}} |H_n| (e^{2\sigma_0 u_n^-} - 1) \,dx \\
&\le \int_{\{u_n < 0\}} (\sigma_0 |\nabla u_n|^2 + \lambda f) (e^{2\sigma_0 u_n^-} - 1) \,dx.
\end{align*}
Combining these estimates:
\begin{multline*}
    2\sigma_0 \int_{\{u_n < 0\}} |\nabla u_n|^2 e^{2\sigma_0 u_n^-} \,dx + \beta \int_{\partial\Omega} u_n^- (e^{2\sigma_0 u_n^-} - 1) \,d\sigma \le \\
    \sigma_0 \int_{\{u_n < 0\}} |\nabla u_n|^2 (e^{2\sigma_0 u_n^-} - 1) \,dx + \lambda \int_\Omega f (e^{2\sigma_0 u_n^-} - 1) \,dx.
\end{multline*}
We perform the change of variables $w_n = e^{\sigma_0 u_n^-}$. This leads to an inequality formally identical to \eqref{mainineqcv} but for $w_n$:
\[
    \frac{1}{\sigma_0} \int_\Omega |\nabla w_n|^2 dx + \sigma_0 \int_{\{u_n < 0\}} |\nabla u_n|^2 dx + \frac{\beta}{\sigma_0} \int_{\partial\Omega} (w_n^2-1) \log w_n \, d\sigma \le \lambda \int_\Omega f w_n^2 dx.
\]
By the same spectral argument used for $v_n$ ({with the same $\gamma$}), we obtain analogous bounds for the gradient and the weighted mass:
\begin{equation}
\label{ctp}
\int_{\Omega} |\nabla w_{n}|^{2} dx \le C, \quad \int_{\Omega} f w_{n}^{2} dx \le C.
\end{equation}

\textit{Conclusion.}
Being $f\in M^{\nicefrac{N}{2}}(\Omega)$, by Proposition \ref{propequiv}, using the estimates \eqref{stimaDv22}, \eqref{stimav2} and \eqref{ctp}, we obtain
\begin{equation}\label{bound vn H1}
\|v_{n}\|_{H^{1}(\Omega)} \le C, \quad \|w_{n}\|_{H^{1}(\Omega)} \le C.
\end{equation}
Finally, the boundedness of the $H^1(\Omega)$ norm implies the boundedness of the trace norm, by using the trace inequality
\begin{equation*}\label{boundvn bordo}
    \|v_{n}\|_{L^2(\partial\Omega)} \le C_{tr} \|v_{n}\|_{H^1(\Omega)}.
\end{equation*}

By the above estimates, up to a subsequence, $u_n$ converges weakly in $H^1(\Om)$, strongly in $L^2(\Om)$ {and in $L^2(\partial\Omega)$,} and a.e. in $\Om$ to some $u\in H^1(\Om)$.

\textit{Step 2: Convergence of $u_{n}$} {In order to prove that $u$ is a weak solution of \eqref{generalrobinpb},
we need the strong convergence of the gradients of $u_n$ in $L^{2}(\Omega)$. For this purpose we follow the strategy in \cite{fm98,fm00,femu13},
which goes back to the technique developed in \cite{BeBoMu} and consisting in
proving the almost everywhere convergence of the gradient
of truncation at height $k$ of $u_n$ after having proved an estimate on the
$L^p$ norms of the gradients of $u_n$ on the sets where $u_n$ is greater than $k$.

{For $k>0$, let $T_k:\R\to\R$ be defined as:
\[
T_k(s)=
\left\{
\begin{array}{ll}
k & \hbox{ if } s\geq k\\
s & \hbox{ if } -k<s<k\\
-k & \hbox{ if } s\leq -k
\end{array}
\right.
\]
and $G_k(s)=s-T_k(s)$.}

Let $Z_n = e^{\sigma_0 |u_n|}\in H^1(\Omega)$. Note that $Z_n$ is controlled by $v_nw_n$,
so it is uniformly bounded in $H^1(\Om)$.
Recalling that $|\nabla u_n| = \frac{1}{\sigma_0} e^{-\sigma_0 |u_n|} |\nabla Z_n|$, we get:
\begin{equation*}\label{code 1}
\int_{\{|u_n|>k\}}|\nabla u_n|^2=
\frac{1}{\sigma_0^2} \int_{\{|u_n|>k\}}e^{-2\sigma_0 |u_n|}|\nabla Z_n|^2\leq
\frac{e^{-2\sigma_0 k}}{\sigma_0^2}\int_{\{|u_n|>k\}}|\nabla Z_n|^2\leq
\frac{e^{-2\sigma_0 k}}{\sigma_0^2}\int_\Omega |\nabla Z_n|^2
\end{equation*}
and by the uniform bounds on $Z_n$ it follows:
\begin{equation*}\label{code 2}
\int_\Omega |\nabla G_k(u_n)|^2\leq
C e^{-2\sigma_0 k},
\end{equation*}
so that $\displaystyle\int_\Omega |\nabla G_k(u_n)|^2$ goes to $0$
as $k$ tends to $+\infty$ uniformly on $n$.

Let us show now that $\nabla T_k(u_n)$ converges a.e. to $\nabla T_k(u)$
and for this purpose it is enough to show the strong convergence
of $\nabla T_k(u_n)$ in $(L^2(\Omega))^N$.

By {reasoning as in \cite[Lemma 3.5]{bmp84}},
in order to get that $\|\nabla T_k(u_n)-\nabla T_k(u)\|_{L^2(\Omega)}$ goes to $0$,
it is enough to show that
\begin{equation}\label{conv da prov}
\int_\Omega \left[a(x,T_k(u_n),\nabla T_k(u_n))-a(x,{T_k(u_n)},\nabla T_k(u))\right]\cdot \left(\nabla T_k(u_n)-\nabla T_k(u)\right)\to 0
\end{equation}
{We observe that $a(x,T_k(u_n),\nabla T_k(u))\to a(x,T_k(u),\nabla T_k(u))$ strongly in $[L^2(\Omega)]^N$, by \eqref{hp3a} and dominated convergence, while $\nabla T_k(u_n)\rightharpoonup\nabla T_k(u)$ weakly in $[L^2(\Omega)]^N$; hence \eqref{conv da prov} is equivalent to the same condition with $T_k(u)$ in place of $T_k(u_n)$ in the second term.}
Let $\psi:\R\to\R$ be the increasing, $C^1$-function defined as
$\psi(s)=2s\,e^{\Gamma s^2}$, where $\Gamma=\frac{\sigma_{0}^2(1+a_0)^2}{4}$.
Remark that $\psi$ satisfies:
\begin{equation}\label{propr psi}
\psi(0)=0\,\qquad
\psi'-\sigma_{0}(1+a_0)|\psi|\geq 1 \,.
\end{equation}
Using $e^{\sigma_0 |u_n|}\psi\left(T_k(u_n)-T_k(u)\right)$ as test function in \eqref{rbvpapp},
we get:
\begin{eqnarray}\label{conv grad 1}
&& \int_\Omega a(x,u_n,\nabla u_n)\cdot\left[\nabla T_k(u_n)-\nabla T_k(u)\right]e^{\sigma_0 |u_n|}\psi'\left(T_k(u_n)-T_k(u)\right)\\
\nonumber &=&
 \int_\Omega \left[H_n(x,u_n,\nabla u_n) - \text{sgn}(u_n)\sigma_0 a(x,u_n,\nabla u_n)\cdot\nabla u_n\right]e^{\sigma_0 |u_n|}\psi \\
\nonumber &&-
 \beta\int_{\partial\Omega}u_n e^{\sigma_0 |u_n|}\psi \,
\end{eqnarray}
where for sake of brevity we {simply} write $\psi$ {instead of} $\psi\left(T_k(u_n)-T_k(u)\right)$.
{The boundary term goes to $0$ by Vitali's theorem on $\partial\Omega$. Indeed, the traces of $Z_n$ are bounded in $L^2(\partial\Omega)$, so that, for a fixed $\eps\in(0,1)$,
\[
|u_n|e^{\sigma_0|u_n|}=\frac1{\sigma_0}Z_n\log Z_n\le C_\eps Z_n^{1+\eps}
\]
is bounded in $L^{\frac{2}{1+\eps}}(\partial\Omega)$, hence equi-integrable on $\partial\Omega$; moreover, $|\psi(T_k(u_n)-T_k(u))|\le\psi(2k)$ and, since the traces of $u_n$ converge to the trace of $u$ in $L^2(\partial\Omega)$
(and then, up to a subsequence, a.e.) and the trace commutes with $T_k$, $\psi(T_k(u_n)-T_k(u))\to\psi(0)=0$ a.e. on $\partial\Omega$.}
Hence, after setting $\Omega_-=\left\{x\in\Omega:\,|u_n|\leq k\right\}$
and $\Omega_+=\left\{x\in\Omega:\,|u_n|>k\right\}$,
equation \eqref{conv grad 1} becomes:
\begin{eqnarray}\label{conv grad 2}
&& \int_{\Omega_-}
 \left[a(x,T_k(u_n),\nabla T_k(u_n))-a(x,T_k(u_n),\nabla T_k(u))\right]\\
\nonumber &&\quad \cdot
 \left[\nabla T_k(u_n)-\nabla T_k(u)\right]e^{\sigma_0 |u_n|}\psi'\left(T_k(u_n)-T_k(u)\right)\\
\nonumber &&+ \int_{\Omega_-}
 a(x,T_k(u_n),\nabla T_k(u))
 \cdot
 \left[\nabla T_k(u_n)-\nabla T_k(u)\right]e^{\sigma_0 |u_n|}\psi'\left(T_k(u_n)-T_k(u)\right)\\
\nonumber &&+ \int_{\Omega_+}
 a(x,u_n,\nabla u_n)(-\nabla T_k(u))\psi'\left(T_k(u_n)-T_k(u)\right) e^{\sigma_0 |u_n|}\\
\nonumber &=&
 \int_{\Omega_-} \left[
 H_n(x,u_n,\nabla T_k(u_n))-\text{sgn}(u_n)\sigma_0 a(x,T_k(u_n),\nabla T_k(u_n))\cdot\nabla T_k(u_n)\right]e^{\sigma_0 |u_n|}\psi\\
\nonumber &&+
 \int_{\Omega_+} \left[
 H_n(x,u_n,\nabla u_n)-\text{sgn}(u_n)\sigma_0 a(x,u_n,\nabla u_n)\cdot\nabla u_n\right]e^{\sigma_0 |u_n|}\psi\\
\nonumber &&+ o(1)\,.
\end{eqnarray}
Let us analyze the {various} integrals in \eqref{conv grad 2}. First, we observe that
\[
\lim_{n\to +\infty}\int_{\Omega_-}
 a(x,T_k(u_n),\nabla T_k(u))
 \cdot
 \left[\nabla T_k(u_n)-\nabla T_k(u)\right]e^{\sigma_0 |u_n|}\psi'\left(T_k(u_n)-T_k(u)\right)=0,
\]
being $\nabla T_{k}(u_{n})\rightharpoonup \nabla T_{k}(u)$ weakly in $[L^{2}(\Omega)]^N$ and
\[
\chi_{\{|u_{n}|\le k\}}{a(x,T_k(u_n),\nabla T_k(u))}e^{\sigma_0 {|u_n|}}\psi'\left(T_k(u_n)-T_k(u)\right)
\]
strongly convergent in $[L^{2}(\Omega)]^N$.
Second, it holds that
\[
\lim_{n}\int_{\Omega_+}
 a(x,u_n,\nabla u_n)(-\nabla T_k(u))\psi'\left(T_k(u_n)-T_k(u)\right) e^{\sigma_0 |u_n|}=0.
\]
Indeed, $\psi'\left(T_k(u_n)-T_k(u)\right)$ is bounded in $L^\infty(\Omega)$;
moreover the inequalities \eqref{hp3a} and \eqref{bound vn H1} {give} that
$a(x,u_n,\nabla u_n)e^{\sigma_0|u_n|}$ is bounded in $[L^2(\Omega)]^N$,
while $u_n$ strongly converges in $L^2(\Om)$.

On the other hand, also the last term in the right-hand side of \eqref{conv grad 2} goes to $0$ as $n$ tends to $+\infty$. Indeed, recalling that $\psi\left(T_k(u_n)-T_k(u)\right)\text{sgn}(u_n)\geq 0$ in $\Omega_+$, and using \eqref{hpH} and \eqref{hp1a}, we have:
\begin{eqnarray*}
&&\int_{\Omega_+} \left[
 H_n(x,u_n,\nabla u_n)-\text{sgn}(u_n)\sigma_0 a(x,u_n,\nabla u_n)\cdot\nabla u_n\right]
 e^{\sigma_0 |u_n|}\psi\left(T_k(u_n)-T_k(u)\right)\\
&=&\int_{\Omega_+} \left[
 H_n(x,u_n,\nabla u_n)\text{sgn}(u_n)-\sigma_0 a(x,u_n,\nabla u_n)\cdot\nabla u_n\right]
 e^{\sigma_0 |u_n|}\psi\left(T_k(u_n)-T_k(u)\right)\text{sgn}(u_n)\\
&\leq&\int_{\Omega_+}
\left[
 \sigma_0|\nabla u_n|^2-\sigma_0 a(x,u_n,\nabla u_n)\cdot\nabla u_n\right]
 e^{\sigma_0 |u_n|}\psi\left(T_k(u_n)-T_k(u)\right)\text{sgn}(u_n)\\
&&\qquad\qquad+\int_{\Omega_+}
\lambda f e^{\sigma_0 |u_n|}\psi\left(T_k(u_n)-T_k(u)\right)\text{sgn}(u_n)\\
&\leq&\int_{\Omega_+}
\lambda f e^{\sigma_0 |u_n|}\psi\left(T_k(u_n)-T_k(u)\right)\text{sgn}(u_n)
\end{eqnarray*}

and we conclude {by Vitali's theorem, since $|\psi(T_k(u_n)-T_k(u))|\le\psi(2k)$ tends to $0$ a.e. and, by \eqref{holdlor}, \eqref{sobolor} and \eqref{bound vn H1}, for every measurable set $E\subset\Omega$,
\[
\int_E fe^{\sigma_0|u_n|}dx\le\|f\|_{\frac N2,\infty,\Omega}\|Z_n\chi_E\|_{\frac{N}{N-2},1,\Omega}\le C|E|^{\frac{N-2}{2N}}\|f\|_{\frac N2,\infty,\Omega}\sup_n\|Z_n\|_{H^1(\Omega)}.
\]}}
Hence, by \eqref{conv grad 2}, using assumptions \eqref{hp3a} and \eqref{hpH},
on $\Omega_-$ we infer:
\begin{equation}\label{conv grad 3}
\begin{array}{lll}
&& \ds\int_{\Omega_-}
 \left[a(x,T_k(u_n),\nabla T_k(u_n))-a(x,T_k(u_n),\nabla T_k(u))\right]\\[.2cm]
 &&\ds\qquad\qquad\qquad \cdot
 \left[\nabla T_k(u_n)-\nabla T_k(u)\right]e^{\sigma_0 |u_n|}\psi'\left(T_k(u_n)-T_k(u)\right)\\[.2cm]
 &\leq&\ds
 \int_{\Omega_-}
 \left[\sigma_0|\nabla T_k(u_n)|^2+\lambda f + \sigma_0(a_0|\nabla T_k(u_n)|+a_1)|\nabla T_k(u_n)|\right] e^{\sigma_0 |u_n|}|\psi|+o(1)\\[.3cm]
 &\leq&\ds
 \int_{\Omega_-}
 \sigma_{0}(1+a_0)|\nabla T_k(u_n)|^2 e^{\sigma_0 |u_n|}|\psi| + \int_{\Omega_-}(\lambda f + \sigma_0 a_1|\nabla T_k(u_n)|)e^{\sigma_0 |u_n|}|\psi|\\[.2cm]
 &&+o(1)\,.
\end{array}
\end{equation}
Reasoning as before, since $|\psi(T_k(u_n)-T_k(u))| \to 0$ a.e.,
the last integral of the last inequality goes to $0$ as $n\to +\infty$
{(recall that $e^{\sigma_0|u_n|}\le e^{\sigma_0k}$ in $\Omega_-$, $f\in L^1(\Omega)$ and $\nabla T_k(u_n)$
is bounded in $[L^2(\Omega)]^N$)}.
By assumption \eqref{hp1a} and \eqref{conv grad 3} we have:
\begin{equation}\label{conv grad 4}
\begin{array}{lll}
&&\ds \int_{\Om_-}
 \left[a(x,T_k(u_n),\n T_k(u_n))-a(x,T_k(u_n),\n T_k(u))\right]\\[.2cm]
 &&\quad\qquad \cdot
 \left[\n T_k(u_n)-\n T_k(u)\right]e^{\sigma_0|u_n|}\psi'\left(T_k(u_n)-T_k(u)\right)\\[.3cm]
 &\leq&
 \ds\int_{\Om_-}
 a(x,T_k(u_n),\n T_k(u_n))\cdot\n T_k(u_n) e^{\sigma_0|u_n|}\sigma_0(1+a_0)|\psi|+o(1)\\[.3cm]
 &=& \ds \int_{\Om_-}
\left[a(x,T_k(u_n),\n T_k(u_n))-a(x,T_k(u_n),\n T_k(u))\right]\\[.3cm]
 &&\quad \qquad\ds \cdot
 \left[\n T_k(u_n)-\n T_k(u)\right]e^{\sigma_0|u_n|}\sigma_0(1+a_0)|\psi|\\[.3cm]
 &&+\ds
 \int_{\Om_-}
 a(x,T_k(u_n),\n T_k(u))\cdot\n \left(T_k(u_n)-T_k(u)\right) e^{\sigma_0|u_n|}\sigma_0(1+a_0)|\psi|\\[.3cm]
 &&+\ds
\int_{\Om_-}
 a(x,T_k(u_n),\n T_k(u_n))\cdot{\n T_k(u)\,e^{\sigma_0|u_n|}}\sigma_0(1+a_0)|\psi|\\[.3cm]
 &&+ o(1)\,.
\end{array}
\end{equation}
The second and third terms in the right-hand side in \eqref{conv grad 4} go to $0$ as $n$ tends to $+\infty$
and hence, using \eqref{propr psi} and \eqref{hp2a},
we get:

\begin{eqnarray*}
0&\le&\int_{\Om_-} \left[a(x,T_k(u_n),\n T_k(u_n))-a(x,T_k(u_{n}),\n T_k(u))\right]\cdot \left(\n T_k(u_n)-\n T_k(u)\right)\\
\nonumber &\leq&
\int_{\Om_-}
 \left[a(x,T_k(u_n),\n T_k(u_n))-a(x,T_k(u_n),\n T_k(u))\right]\\
\nonumber &&\quad \qquad\cdot
 \left[\n T_k(u_n)-\n T_k(u)\right]e^{\sigma_0|u_n|}
 \left[\psi'-\sigma_0(1+a_0)|\psi|\right]\\
\nonumber &\le&o(1)\,.
\end{eqnarray*}
On the other hand, on $\Omega_+$ we have $\nabla T_k(u_n)=0$, so that, by \eqref{hp3a},
\begin{multline*}
0\le\int_{\Omega_+}\left[a(x,T_k(u_n),0)-a(x,T_k(u_{n}),\n T_k(u))\right]\cdot\left(-\n T_k(u)\right)\\
\le\int_\Omega\left(a_0|\nabla T_k(u)|+2a_1\right)|\nabla T_k(u)|\,\chi_{\{|u_n|>k\}}\,dx\to0,
\end{multline*}
by dominated convergence, since $\chi_{\{|u_n|>k\}}\nabla T_k(u)\to0$ a.e. in $\Omega$. Adding the two inequalities we obtain
\eqref{conv da prov}.

In order to conclude the proof, it is enough to note that
\[
\nabla u_n-\nabla u=\nabla T_k(u_n)-\nabla T_k(u)+\nabla G_k(u_n)-\nabla G_k (u),
\]
so that, by the strong convergence in $[L^2(\Omega)]^N$ of $\nabla T_k(u_n)$ and the estimate on $\int_\Omega|\nabla G_k(u_n)|^2dx$ proved above,
\[
\limsup_{n\to+\infty}\|\nabla u_n-\nabla u\|_{L^2(\Omega)}\le\sup_{n}\|\nabla G_k(u_n)\|_{L^2(\Omega)}+\|\nabla G_k(u)\|_{L^2(\Omega)}\le Ce^{-\sigma_0k}+\|\nabla G_k(u)\|_{L^2(\Omega)}
\]
for every $k>0$. Letting $k\to+\infty$, we get the strong convergence in $[L^{2}(\Omega)]^N$ of $\nabla u_{n}$.

This finally allows to pass to the limit in the approximating problems \eqref{rbvpapp},
obtaining that $u$ is a solution to \eqref{generalrobinpb}.
\end{proof}

We conclude the paper by listing some radial {examples}.
\begin{example}\label{ex:ball}
We consider the problem with $\lambda,\beta>0$, in the unit ball $B_1(0) \subset \mathbb{R}^N$:
\begin{equation*} \label{eq:prob_u}
\begin{cases}
    \Delta u + |\nabla u|^2 + \lambda = 0 & \text{in } B_1, \\
    \frac{\partial u}{\partial \nu} + \beta u = 0 & \text{on } \partial B_1,
\end{cases}
\end{equation*}
and consider solutions $u\in H^{1}({B_1})$ such that $e^{u}\in H^{1}({B_1})$.

Using the transformation $v = e^u$, the problem is equivalent to the following system for the positive function $v$:
\begin{equation} \label{eq:prob_v}
\begin{cases}
    \Delta v + \lambda v = 0 & \text{in } B_1, \\
    \frac{\partial v}{\partial \nu} + \beta v \log v = 0 & \text{on } \partial B_1.
\end{cases}
\end{equation}
We look for radial solutions $v(r)$ with $r=|x|$. The equation in (\ref{eq:prob_v}) becomes:
\[
v'' + \frac{N-1}{r}v' + \lambda v = 0, \quad r \in (0,1),
\]
and the boundary condition at $r=1$ reads:
\begin{equation} \label{eq:bc_radial}
v'(1) + \beta v(1) \log(v(1)) = 0.
\end{equation}

Let $\mu = \sqrt{\lambda}$. We impose $\lambda<(j_{\frac{N-2}{2},1})^{2}$, where $j_{\frac{N-2}{2},1}$ is the first positive zero of $J_{\frac{N-2}{2}}(x)$.

 The solution is given by:
\[
v(r) = A \, r^{-\frac{N-2}{2}} J_{\frac{N-2}{2}}(\mu r),\quad r\in (0,1)
\]
where $J_\nu$ denotes the Bessel function of the first kind of order $\nu$, and $A > 0$ is a constant to be determined. Let $\alpha = \frac{N-2}{2}$. We define $\Phi(r) = r^{-\alpha} J_{\alpha}(\mu r)$.
The boundary condition (\ref{eq:bc_radial}) implies:
\[
A \Phi'(1) + \beta A \Phi(1) \log(A \Phi(1)) = 0.
\]
Being $\beta > 0$ and $\Phi(1) \neq 0$, we solve for $A$:
\[
\log A = - \frac{\Phi'(1)}{\beta \Phi(1)} - \log \Phi(1).
\]
Using the recurrence relation $\frac{d}{dz}(z^{-\alpha}J_\alpha(z)) = -z^{-\alpha}J_{\alpha+1}(z)$, we have:
\[
\Phi'(1) = -\mu J_{\alpha+1}(\mu), \quad \Phi(1) = J_{\alpha}(\mu).
\]
Substituting back into $u(r) = \log v(r) = \log A + \log \Phi(r)$, we obtain the explicit solution:
\[
u(r) = \frac{\mu}{\beta} \frac{J_{\frac{N}{2}}(\mu)}{J_{\frac{N-2}{2}}(\mu)} + \log \left( \frac{r^{-\frac{N-2}{2}} J_{\frac{N-2}{2}}(\mu r)}{J_{\frac{N-2}{2}}(\mu)} \right),\quad r\in(0,1).
\]
\end{example}

\begin{example}\label{ex:hardynat} We analyze the problem with a singular potential term, given by the PDE:
\begin{equation*}
\begin{cases}
    \Delta u + |\nabla u|^2 + \frac{\lambda}{|x|^2} = 0 & \text{in } B_1\\
    \frac{\partial u}{\partial \nu} + \beta u = 0 & \text{on }\partial B_1
    \end{cases}
\end{equation*}
with $\beta,\lambda>0$. By changing the variable $v=e^u$ and looking for radially symmetric solutions $v(x)=R(r)$, we obtain the ODE
\begin{equation*}
    R''(r) + \frac{N-1}{r} R'(r) + \frac{\lambda}{r^2} R(r) = 0
\end{equation*}
We look for a solution of the form $R(r)=Ar^\alpha$. This leads to
\begin{equation*}
    \alpha^2 + (N-2)\alpha + \lambda = 0
\end{equation*}
{that is, $\lambda=\gamma(N-2-\gamma)$ with $\gamma=-\alpha$. Hence, as in Example 2, the roots $\alpha_{1}\ge\alpha_2$ are real and negative if and only if $\lambda\le\frac{(N-2)^2}{4}=\lambda_{1,f}(B_1)$, with $f=|x|^{-2}$, and $R$ is, up to a constant, the eigenfunction $|x|^{-\gamma_i}$ of Example 2, with $\gamma_i=-\alpha_i$.}

The single unknown constant $A$ is determined by the boundary condition at $r=1$:
\[
\frac{R'(1)}{R(1)} + \beta \log(R(1)) = 0
\]
which gives $\alpha_{{i}} + \beta\log(A) = 0$.
Being $\beta > 0$, {we obtain the two positive solutions}
\begin{equation*}
    u_{{i}}(x) = \alpha_{{i}} \left( \log(|x|) - \frac{1}{\beta} \right){,\qquad i=1,2.}
\end{equation*}

The functions $u_i$ are solutions in $H^1(B_1)$ for every $\lambda\le\frac{(N-2)^2}{4}$. As in Example 2, $e^{u_i}=A|x|^{\alpha_i}\in H^1(B_1)$ if and only if $\gamma_i<\frac{N-2}{2}$, that is, only for $i=1$ and $\lambda<\frac{(N-2)^2}{4}$. Therefore, for $\lambda<\frac{(N-2)^2}{4}$, $u_1$ is the solution given by Theorem \ref{thm:main_existence}, while $u_2$ is a second solution, with $e^{u_2}\notin H^1(B_1)$; for $\lambda=\frac{(N-2)^2}{4}$, $u_1=u_2$ is a solution in $H^1(B_1)$ with $e^{u_1}\notin H^1(B_1)$; for $\lambda>\frac{(N-2)^2}{4}$, no solution in $H^1(B_1)$ exists, by Proposition \ref{prop:necessary}.

\end{example}

\section*{Acknowledgement}
This work has been partially supported by PRIN PNRR 2022 ``Linear and Nonlinear PDEs: New directions and Applications'', by GNAMPA of INdAM, and by the ``Geometric-Analytic Methods for PDEs and Applications" project - funded by European Union - Next Generation EU  within the PRIN 2022 program (D.D. 104 - 02/02/2022 Ministero dell'Universit\`{a} e della Ricerca). This manuscript reflects only the authors' views and opinions and the Ministry cannot be considered responsible for them.


\begin{thebibliography}{10}

\bibitem{abddaper} B. Abdellaoui, A. Dall'Aglio, I. Peral,
\newblock Some remarks on elliptic problems with critical growth on the gradient.
\newblock {\em J. Differential Equations}, 222(1), 2006, 21--62.

\bibitem{adi} Adimurthi,
\newblock Hardy Sobolev inequality in $H^{1}(\Omega)$ and its applications,
{\em Commun. Contemp. Maths}, 4, 2002, 409--439.

\bibitem{adiest} Adimurthi, M.J. Esteban,
\newblock An improved Hardy-Sobolev inequality in $W^{1,p}(\Omega)$ and its applications to Schr\"odinger operators,
\newblock {\em NoDEA Nonlinear Diﬀerential Equations Appl.}, 12, 2005, 243--263.

\bibitem{alv} A. Alvino,
\newblock Sulla diseguaglianza di Sobolev in spazi di Lorentz,
\newblock {\em Boll. Unione Mat. Ital. A}, 14, 1977, 148--156.

\bibitem{afm} A. Alvino, V. Ferone, A. Mercaldo,
\newblock Sharp a priori estimates for a class of nonlinear elliptic
equations with lower order terms,
\newblock {\em Annali di Matematica Pura Appl.} 194, 2015, 1169–1201.

\bibitem{BeBoMu} A. Bensoussan, L. Boccardo, F. Murat,
\newblock On a non linear partial differential equation having natural growth terms and unbounded solutions,
\newblock {\em Ann. Inst. H. Poincar\'e Anal. Non lin\'eaire}, 5, 1988, 347--364.

\bibitem{bmp84} {L. Boccardo, F. Murat, J.P. Puel, R\'esultats d'existence pour certains probl\`emes elliptiques quasilin\'eaires, {\em Ann. Sc. Norm. Sup. Pisa Cl. Sci. (4)} 11 (1984), 213--235.}

\bibitem{bmp92} L. Boccardo, F. Murat, J.P. Puel,
\newblock {$L^\infty$} estimate for some nonlinear elliptic partial
  differential equations and application to an existence result.
\newblock {\em SIAM J. Math. Anal.}, 23(2), 1992, 326--333.

\bibitem{brezponce} H. Brezis, A. Ponce, Remarks On The Strong Maximum Principle, Differential and Integral Equations Volume 16, Number 1, January 2003, Pages 1–12

\bibitem{chab} J. Chabrowski, On the nonlinear Neumann problem
involving the critical Sobolev exponent and
Hardy potential, Rev. Mat. Complut. 17(1) (2004) 195–227.

\bibitem{cpr} J. Chabrowski, I. Peral, B. Ruf,
\newblock On an eigenvalue problem involving the Hardy potential,
\newblock {\em Communications in Contemporary Mathematics}, 12(6), 2010, 953--975.

\bibitem{dp} F. Della Pietra,
\newblock Existence results for non-uniformly elliptic equations with general growth in the gradient.
\newblock {\em Differ. Integral Equ.}, 21(9-10), 2008, 821--836.

\bibitem{dpdb}  Della Pietra F., di Blasio G., \textit{Blow-up solutions for some nonlinear elliptic equations involving a Finsler-Laplacian}, {Publicacions Matem\`atiques} 61, 2017,  213-238.

\bibitem{dpg4}
F. Della Pietra, N. Gavitone,
\newblock Anisotropic elliptic equations with general growth in the gradient and {H}ardy-type potentials.
\newblock {\em J. Differential Equations}, 255(11), 2013, 3788--3810.

\bibitem{dpper} F. Della Pietra, I. Peral,
\newblock Breaking of resonance for elliptic problems with strong degeneration at infinity.
\newblock {\em Commun. Pure Appl. Anal.}, 10(2), 2011, 593--612.

\bibitem{fm98} V. Ferone, F. Murat,
\newblock Quasilinear problems having quadratic growth in the gradient: an existence result when the source term is small.
\newblock In {\em \'{E}quations aux d\'{e}riv\'{e}es partielles et
  applications}, 497--515. Gauthier-Villars, \'{E}d. Sci. M\'{e}d.
  Elsevier, Paris, 1998.

\bibitem{fm00} V. Ferone, F. Murat,
\newblock Nonlinear problems having natural growth in the gradient: an existence result when the source terms are small.
\newblock {\em Nonlinear Anal.}, 42(7, Ser. A), 2000, 1309--1326.

\bibitem{femu13} V. Ferone, F. Murat,
\newblock Nonlinear elliptic equations with natural growth in the gradient and source terms in Lorentz spaces.
\newblock {\em J. Differential Equations}, 256(3), 2014, 577--608.

\bibitem{gmp12} N. Grenon, F. Murat, A. Porretta,
\newblock A priori estimates and existence for elliptic equations with gradient dependent terms.
\newblock {\em Ann. Sc. Norm. Super. Pisa Cl. Sci. (5)}, 14(1), 2015, 137--205.

\bibitem{hmv99} K. Hansson, V.G. Maz'ya, and I.E. Verbitsky,
\newblock Criteria of solvability for multidimensional Riccati equations.
\newblock {\em Ark. Mat.}, 37(1), 1999, 87--120.

\bibitem{kazkra89} J.L. Kazdan, R.J. Kramer,
\newblock Invariant criteria for existence of solutions to second-order quasilinear elliptic equations.
\newblock {\em Comm. Partial Differential Equations}, 14(12), 1989, 1817--1845.

\bibitem{ll} J. Leray, J.L. Lions,
\newblock Quelques r\'esultats de Vi{\v s}ik sur les probl\`emes elliptiques
  non lin\'eaires par les m\'ethodes de {M}inty-{B}rowder.
\newblock {\em Bull. Soc. Math. France}, 93, 1965, 97--107.

\bibitem{lions} P.-L. Lions, The concentration compactness principle in the calculus of variations. The
locally compact case, {\em Ann. Inst. H. Poincar\'e Anal. Non Lin\'eaire} 1 (1984), 109--145 and 223--283.

\bibitem{lions2} {P.-L. Lions, The concentration-compactness principle in the calculus of variations. The limit case, Part 1, {\em Rev. Mat. Iberoamericana} 1 (1985), 145--201.}

\bibitem{oliva} F. Oliva, Existence and uniqueness of solutions to some singular equations with natural growth. Ann. Matematica Pura Appl. (2021) 200:287–314.

\bibitem{smets} D. Smets, A Concentration-Compactness Lemma with Applications
to Singular Eigenvalue Problems, {\em J. Funct. Anal.} 167 (1999), 463--480.

\bibitem{sw}
A. Szulkin, M. Willem, Eigenvalue problems with indefinite weight. Stud. Math. 135, 191–201 (1999).

\bibitem{tr03} C. Trombetti,
\newblock Non-uniformly elliptic equations with natural growth in the gradient,
\newblock {\em Potential Analysis}, 18(4), 2003, 391--404.

\end{thebibliography}
\end{document}